\documentclass[reqno,10pt]{amsart}
\usepackage{amsmath}
\usepackage{amssymb}
\usepackage{amsfonts}

\usepackage{graphicx}
\usepackage[abs]{overpic}
\usepackage{caption}
\usepackage{subcaption}
\usepackage{wrapfig}
\usepackage{float}
\usepackage{graphicx}
\usepackage{physics}
\usepackage{esint} 
\usepackage{braket} 

\newcommand{\HH}{\mathcal{H}} 

\usepackage{tikz}
\usetikzlibrary{decorations.pathreplacing,calc}
\def\centerarc[#1](#2)(#3:#4:#5)
{ \draw[#1] ($(#2)+({#5*cos(#3)},{#5*sin(#3)})$) arc (#3:#4:#5); }

\definecolor{blue_links}{RGB}{13,0,180} 

\usepackage{hyperref}
\hypersetup{
    colorlinks=true, 
    linktoc=all,     
    linkcolor=blue_links,  
    citecolor=blue_links,
    urlcolor=blue_links,
}

\usepackage{mathtools}

\usepackage{xcolor}

\newtheorem{theorem}{Theorem}[section]
\newtheorem{lemma}[theorem]{Lemma}
\newtheorem{proposition}[theorem]{Proposition}

\newtheorem{definition}[theorem]{Definition}

\newtheorem{remark}[theorem]{Remark}

\newtheorem*{theorem*}{Theorem}

\newcommand{\N}{\mathbb{N}}

\newcommand{\R}{\mathbb{R}}
\newcommand{\C}{\mathbb{C}}
\newcommand{\CC}[1]{\C #1:#1}
\newcommand{\A}[1]{\abs{#1}_{\mathbb{C}}}

\newcommand{\BBB}{\color{black}} 
 
\newcommand{\EEE}{\color{black}}

\def\eps{\varepsilon}

\def\XXint#1#2#3{{\setbox0=\hbox{$#1{#2#3}{\int}$}
\vcenter{\hbox{$#2#3$}}\kern-.5\wd0}}

\usepackage{color}

\numberwithin{equation}{section}

\begin{document} 

\title[Strong existence for free discontinuity problems in linear elasticity]{Strong existence for free discontinuity problems\\ in linear elasticity}
\author[M. Friedrich]{Manuel Friedrich} 
\address[Manuel Friedrich]{Department of Mathematics, Friedrich-Alexander Universit\"at Erlangen-N\"urnberg. Cauerstr.~11,
    D-91058 Erlangen, Germany, \& Mathematics M\"{u}nster,  
University of M\"{u}nster, Einsteinstr.~62, D-48149 M\"{u}nster, Germany}
\email{manuel.friedrich@fau.de}
\author[C. Labourie]{Camille Labourie}
\address[Camille Labourie]{Department of Mathematics, Friedrich-Alexander Universit\"at Erlangen-N\"urnberg. Cauerstr.~11,
D-91058 Erlangen, Germany}
\email{camille.labourie@fau.de}
\author[K. Stinson] {Kerrek Stinson} 
\address[Kerrek Stinson]{Hausdorff Center for Mathematics, University of Bonn, Endenicher Allee 62, 53115 Bonn,
Germany}
\email{kerrek.stinson@utah.edu}


\begin{abstract}   
    In this note we show Ahlfors-regularity for a large class of quasiminimizers of the Griffith functional. This allows us to prove that, for a range of free discontinuity problems in linear elasticity with anisotropic, cohesive, or heterogeneous behavior, minimizers have an essentially closed jump set and are thus strong minimizers.
Our notion of quasiminimality is inspired by and generalizes previous notions in the literature for the Mumford-Shah functional, and comprises functions which locally close to the crack have at most a fixed percentage of excess crack relative to minimizers.
    As for minimizers of the Griffith functional treated by Chambolle, Conti, and Iurlano, our proof of Ahlfors-regularity relies on contradiction-compactness and an approximation result for $GSBD$ functions, showing the robustness of this approach with respect to generalization of bulk and surface densities. 
\end{abstract}
\EEE

\maketitle

\section{Introduction}
The equilibrium state of a brittle elastic material can be found by minimizing the Griffith fracture energy, as introduced in the context of linearized elasticity by {\sc Francfort and Marigo} \cite{FM}.
Letting $\Omega\subset \R^N$ be   a \EEE bounded open set with Lipschitz boundary,  representing \EEE the reference configuration of a solid, and $\C$ be a fourth-order elastic tensor, the Griffith energy is defined by
\begin{equation}\label{eqn:GriffithEnergy}
    \int_{\Omega} \! \C e(u):e(u) \dd{x} + \mathcal{H}^{N-1}(J_u),
\end{equation}
for a displacement $u \in GSBD^2  (\Omega) \EEE $, where $e(u)  := \EEE (\nabla u + \nabla u^T)/2$  denotes  the  symmetrized  gradient of $u$, $J_u$ is the associated crack set (i.e., the jump set  of $u$), \EEE and $\mathcal{H}^{N-1}$ is the $(N-1)$-dimensional surface measure. The energy (\ref{eqn:GriffithEnergy}) is a \emph{weak formulation} of the Griffith energy as it does not impose that the crack $J_u$ is closed. The \emph{strong formulation},  instead, \EEE assumes that the cracks are closed  in $\Omega$, and this `closedness'  acts as a  basal \EEE regularity property preventing the crack from being diffusely distributed throughout the material domain.

The goal of this paper is to show that a wide variety of free discontinuity problems in linear elasticity with anisotropic, (non-degenerate) cohesive, or heterogeneous behavior have strong minimizers.
Precisely, in  Theorem \ref{cor:aniso} we show that a $GSBD$ minimizer of the energy
\begin{equation}\label{main example}
    \int_{\Omega} f(x,e(u)) \dd{x} + \int_{J_u} \psi(x,u^+,u^-,\nu_u) \dd{\HH^{N-1}}
\end{equation}
has an essentially closed jump set,  i.e., \EEE $\mathcal{H}^{N-1}(\overline{J_u}\setminus J_u) = 0$, and thereby  is a \EEE minimizer of the strong formulation.
We frame our result as a `weak-to-strong' improvement in regularity for minimizers, given basic assumptions on the densities, i.e., $\psi \colon \Omega \times \R^N \times \R^N \times  \mathbb{S}^{N-1} \EEE \to  \R_+ \EEE$ and $f \colon \Omega \times \R^{N \times N}_{\mathrm{sym}} \to \R_+$ are Borel functions satisfying $M^{-1} \leq \psi \leq M$, and
\begin{equation*}
    \abs{f(x,\xi) - \C \xi:\xi} \leq M \left(1 + \abs{\xi}^{q}\right) \quad \text{ for all } x\in \Omega \text{ and } \xi \in \R_{\mathrm{sym}}^{N \times N},
\end{equation*}
for a fixed constant $M \geq 1$, an elastic tensor $\C$, and   an \EEE exponent $q \in (0,2)$. 
(Here, $\R^{N \times N}_{\mathrm{sym}} \subset \R^{N \times N}$ denotes the subset of symmetric matrices, $\mathbb{S}^{N-1} := \lbrace x\in \R^N \colon \, |x| = 1 \rbrace$,  and $\R_+ := [0,+\infty)$. \BBB Actually, we can also treat a slightly more general setting, cf.\ \eqref{eq_gxi}--\eqref{eq_gxi2} below.) \EEE  Our results also hold for local minimizers, i.e., minimizers with respect to compact perturbations on balls. \EEE  
Note that in this paper we do not address existence of  (weak) \EEE minimizers for \eqref{main example}  as this has been treated thoroughly  in the literature by \EEE {studying} compactness \cite{DalMaso:13, CC2020} and separately the lower semicontinuity of the bulk \cite{CC2020} and the \BBB surface term \cite{perugini, FPS}, \EEE see Remark \ref{rmk:...manuel}. \EEE

Let us mention that for functionals with control on the \emph{full} gradient and not only on its symmetrized part, some strong existence results generalizing the ones for the Mumford--Shah functional \cite{DCL} are  available, namely for energies with specific \BBB  anisotropic bulk densities  \cite{new-ref1, Fonseca-Fusco, new-ref2, new-ref3}  \EEE covering {both the scalar and the vectorial case}, {and for} surface energies with dependence on the \EEE traces \cite{BL}, in connection to  free {discontinuity} problems with Robin conditions. In this work, we extend these results to a large class of admissible bulk and surface densities, particularly requiring only a control on the symmetrized gradient.

\EEE

In dimension $N=2$ \cite{CFI} and $N\geq 2$  \cite{CCI, CC}, it has been shown that   minimizers of the Griffith energy (\ref{eqn:GriffithEnergy})  have an essentially closed jump set and  are minimizers of the strong formulation.
Practically, this  is achieved by proving that  the crack is locally \emph{Ahlfors-regular}, meaning  that \EEE it satisfies the \emph{uniform density estimate} 
\begin{equation}\label{eqn:AhlfIntro}
    C_{\rm Ahlf}^{-1} r^{N-1} \leq \HH^{N-1}\big(J_u \cap B(x_0,r)\big) \leq C_{\rm Ahlf} r^{N-1}
\end{equation}
in any ball $B(x_0,r) \subset \Omega$ centered at {$x_0\in \overline{J_u}$}  with radius $0 < r \leq r_0$, for some uniform constants $r_0 > 0$ and $C_{\rm Ahlf} \geq 1$. \EEE

To  prove  that minimizers of \eqref{main example} have closed crack, we will show that property  \eqref{eqn:AhlfIntro}  is inherited by more than just   minimizers \EEE of the energy \eqref{eqn:GriffithEnergy}, and extends to a flexible class of functions called \emph{quasiminimizers}. Then, we simply show that  (local) minimizers \EEE of \eqref{main example} are specific examples of quasiminimizers. This larger class of functions throws away the unnecessary information carried by the densities $f$ and $\psi$, and states that (up to a deviation/gauge) $u$ is an $M$-quasiminimizer if for all (local) competitors $v$ it holds that
$$\int_{\Omega} \CC{e(u)} \dd{x} + M^{-1}\mathcal{H}^{N-1}(J_u) \leq \int_{\Omega} \CC{\BBB e(v) \EEE } \dd{x} + M\mathcal{H}^{N-1}(J_v), $$
see Definition~\ref{defi_quasi}.
Formally, a function is a quasiminimizer if, away from the crack, it has near minimal elastic energy and, close to the crack, it (locally) has at most a fixed percentage of excess crack relative to a local minimizer of (\ref{eqn:GriffithEnergy}).
We emphasize that, in contrast to minimizers, \EEE the crack set of a quasiminimizer is in general  not smooth for reasons of bi-Lipschitz invariance, see Remark \ref{rmk:bilip}. Still,  as we show in this paper, \EEE   it satisfies \EEE  some \EEE quantitative properties with uniform bounds,  namely \EEE Ahlfors-regularity.

Our analysis will follow two parallel notions of quasiminimizer: one in the $GSBD$ setting and one in the classical setting of pairs $(u,K)$ where $K$ is a closed set and $u$ is locally Sobolev outside of $K$. For quasiminimizers in {both} settings, we will show that the crack is Ahlfors-regular and satisfies a density estimate of the form (\ref{eqn:AhlfIntro}), see Theorem \ref{thm_af}. Note that, while in the case of pairs the crack $K$ is already closed,  \BBB the Ahlfors-regularity \eqref{eqn:AhlfIntro} is important and provides additional information. In particular, it is the starting point for studying finer regularity properties of crack sets  of    \EEE minimizers, see \cite{FLS}. 

To obtain Ahlfors-regularity of the crack set, our approach is motivated by the seminal paper of {\sc De Giorgi et al.} \cite{DCL}, wherein a contradiction-compactness argument is used to show that in regions with small crack, the energy in a ball $B(x_0,r)$ decays on the order of $r^{N}$ (like a bulk energy), much faster than the surface energy scaling $r^{N-1}$. Accounting for serious technical difficulties arising from the vectorial nature of the Griffith energy, the contradiction-compactness was also used by {\sc Chambolle et al.} \cite{CCI} for the Griffith energy   (see also \cite{CFI} for the two-dimensional case). \EEE Therein, an approximation result for $GSBD$ functions with small crack was introduced, which will also be an essential tool in our argument, see Lemma \ref{lem_compactness}. Going beyond these results, our argument will show that the contradiction-compactness argument is robust and able to absorb large perturbations in the surface energy.

Informed by the work of {\sc Bonnet} on the Mumford--Shah functional \cite{Bonnet}, where it was shown that blow-up limits are only minimal with respect to competitors satisfying topological constraints, we also cover a weaker minimality condition called \emph{topological quasiminimality},  see Definition \ref{def: main def}.
In this setting, Theorem \ref{thm_af} provides the Griffith counterpart of {\sc Siaudeau}'s result \cite{Si}, which considers topological quasiminimizers in the scalar setting (see also \cite[Chapter 72]{David}).
As an application, our result proves Ahlfors-regularity of the crack set for Griffith global minimizers  in $\R^N$, \EEE see \cite[Definition 7.1]{LL2}.
A technical hurdle will be to adapt the aforementioned approximation result introduced in \cite{CCI} to the case of admissible pairs, which always require closed cracks.
Furthermore, we will have to investigate limits of sequences with vanishing jump sets under the weaker assumption of topological quasiminimality.
\EEE

Finally, we remark that our notion of quasiminimality is inspired by and generalizes \cite[Definition~7.21]{David} and \cite[Definition~2.1]{BL} used in the Mumford--Shah setting, see Remark \ref{rmk:otherQuasi}.
In particular, {\sc Bucur and Luckhaus} \cite{BL} obtain Ahlfors-regularity of scalar almost quasiminimizers by deriving a monotonicity formula, which is interesting as an alternative approach to the {\sc De Giorgi et al.} \cite{DCL} framework, but it is not clear how such a monotonicity formula could be found in the vectorial setting of linearized elasticty.

The  paper \EEE is organized as follows. In Section \ref{sec:mainresults}, we introduce the definition of quasiminimizer and state the main results: Ahlfors-regularity for quasiminimizers and existence of strong solutions for free discontinuity problems. Here we show that the latter is a consequence of the former. In Section \ref{sec:density lower bound}, we prove that the Griffith energy decays like a bulk energy in regions with small crack and use this to show that local minimizers of \eqref{main example} \BBB are specific examples of quasiminimizers of \eqref{eqn:GriffithEnergy}, hence guaranteeing an \EEE Ahlfors-regular crack as in (\ref{eqn:AhlfIntro}). 
The decay of the energy relies on an approximation result whose proof is discussed in Section \ref{sec:Approx}.

\section{Main results}\label{sec:mainresults}

\subsection{Notation}
 For  $N \geq 2$ let $\Omega\subset\R^N$ be an open set.
We say that a constant is universal if it only depends on $N$. We  use $|\cdot|$ for the Euclidean norm of a vector,  and \EEE given two matrices $\xi, \eta \in \R^{N \times N}$, the notation $\xi : \eta$ denotes the Frobenius inner product of $\xi$ and $\eta$,  i.e., \EEE
\begin{equation*}
    \xi : \eta := \sum_{i,j=1}^N \xi_{ij} \eta_{ij}.
\end{equation*}
The Frobenius norm is given by $\abs{\xi} := \sqrt{\xi : \xi }$. Given $a, b \in \R^N$,  we let \EEE $a \odot b := ((a_i b_j + b_i a_j)/2)_{ij} \in \R^{N \times N}$.  For \EEE the entire paper, we fix a coercive fourth-order tensor $\C \colon \R^{N \times N} \to \R^{N \times N}$ such that 
\begin{equation*}
    \C(\xi - \xi^T) = 0 \quad \text{and} \quad \C \xi : \xi \geq   c_0 \EEE   \abs{\xi + \xi^T}^2 \quad \quad \text{for all $\xi \in \R^{N \times N}$}
\end{equation*}
for some constant $c_0  > 0 \EEE $.
Note that $\C$ defines a scalar product on the space $\R^{N \times N}_{\mathrm{sym}}$ of symmetric matrices. 
We say that a linear map $\R^N \ni x \mapsto a(x):= Ax + b$, where $A\in \mathbb{R}^{N\times N}$ and $b\in \R^N$, is a \emph{rigid motion} if $A$ belongs to $\R^{N\times N}_{\rm skew}$, the set of skew symmetric matrices.
We use $B(x,r)$ for the open ball of radius $r>0$ centered at $x\in \mathbb{R}^N.$ We denote by $\omega_N$ and $\sigma_{N-1}$ the volume of the $N$-dimensional unit ball and the surface area of the $(N-1)$-dimensional sphere,  respectively.  By \EEE an abuse of notation, {we also denote by $|\cdot|$ the Lebesgue measure,} and  we write \EEE $\mathcal{H}^{N-1}$ for the $(N-1)$-dimensional Hausdorff measure.  The symmetric difference of two sets $V_1$ and $V_2$ is denoted by $V_1 \triangle V_2$. \EEE {For a set $V$, we denote the associated characeristic function by $\chi_V$.}
For a set of finite perimeter $V$, $\partial^*V$ denotes the reduced boundary, see \cite[ Definition~2.60\EEE]{AFP}.

In the following, we consider functions in the \EEE space of \emph{generalized special functions of bounded deformation}, denoted by $GSBD(\Omega)$, which consists of measurable maps $u\colon \Omega \to \R^N$ having integrable symmetric derivative $e(u)$ and a rectifiable jump set $J_u$, where both of these quantities are interpreted via slicing. We refer to \cite{DalMaso:13} for the precise definition. We denote by $GSBD^2(\Omega)$ those functions $u$ belonging to $GSBD(\Omega)$ such that $e(u) \in L^2(\Omega;\R^{N\times N})$ and $\mathcal{H}^{N-1}(J_u) < \infty.$ Naturally, $GSBD^2_{\rm loc}(\Omega)$ consists of functions belonging to $GSBD^2(\Omega')$ for all $\Omega' \subset \subset \Omega.$

\subsection{Quasiminimizers and Ahlfors-regularity}
We introduce our notion  of \EEE \emph{quasiminimality}, which requires a function to satisfy a minimality condition with respect to an appropriate class of competitors.

We introduce different notions of competitors for functions $u \in GSBD_{\mathrm{loc}}(\Omega)$ and for pairs $(u,K)$, consisting of a relatively closed subset $K \subset \Omega$ and a function $u \in W^{1,2}_{\mathrm{loc}}(\Omega \setminus K;\R^N)$.
\EEE

\begin{definition}[Competitors]\label{def:com}  Let $x_0\in \Omega$ and $r>0$ be such that $B(x_0,r)\subset\subset \Omega.$
\begin{enumerate}
\item[1)] A function $v \in GSBD^2_{\mathrm{loc}}(\Omega)$ is a \emph{$GSBD$ competitor} of $u$ in $B(x_0,r)$ if $\set{v \ne u} \subset \subset B(x_0,r)$.

\item[2)] 
A pair $(v,L)$ in $\Omega$ is a \emph{competitor} of $(u,K)$ in $B(x_0,r)$ if
\begin{equation}\label{eqn:compCond}
    L \setminus B(x_0,r) = K \setminus B(x_0,r) \quad \text{and} \quad v = u \quad \text{a.e. in} \quad \Omega \setminus \big(K \cup B(x_0,r)\big).
\end{equation}
\item[3)] 
A pair $(v,L)$ in $\Omega$ is a \emph{topological competitor} of $(u,K)$ in $B(x_0,r)$ if
(\ref{eqn:compCond}) holds and in addition  the following condition {is satisfied}: \EEE
\begin{equation}\label{eq_topological}
    \text{all $x,y \in \Omega \setminus \big(K \cup B(x_0,r)\big)$ that are separated by $K$ are also separated by $L$}.
\end{equation}
This means that, if $  x, y \EEE \in \Omega \setminus (K \cup B(x_0,r))$ belong to different connected components of $\Omega \setminus K$, they also belong to different connected components of $\Omega \setminus L$.
\end{enumerate}
\end{definition}

We say that a pair $(u,K)$ is \emph{coral} provided that for all $x_0 \in K$ and for all $r > 0$ we have
\begin{equation}\label{eq_coral}
    \HH^{N-1}\big(K \cap B(x_0,r)\big) > 0.
\end{equation}
For a given pair $(u,K)$, there is always a coral representative. In particular, such a representative gets rid of any isolated points of $K$ which do not affect the energy but would prevent an estimate of the form (\ref{eqn:AhlfIntro}) from holding at all points in the crack.  We say a pair $(u,K)$ has \emph{locally finite energy} if $\Vert e(u) \Vert_{L^2(B(x_0,r))}  + \mathcal{H}^{N-1}(K \cap B(x_0,r)) <\infty $ for all $B(x_0,r)\subset\subset \Omega.$ \EEE  We define a \emph{gauge} as a nondecreasing function $h \colon (0,+\infty) \to [0,+\infty]$ with $\limsup_{r \to 0} h(r) < \infty$.

\begin{definition}[Quasiminimizers]\label{def: main def} 
Let $M \geq 1$ be a constant and let $h$ be a gauge. 
\begin{enumerate}
    \item[1)] \label{defi_quasi}
    A function $u \in GSBD^2_{\mathrm{loc}}(\Omega;\R^N)$ is a \emph{ $GSBD$ \EEE local $M$-quasiminizer} with gauge $h$ in $\Omega$ if for any open ball $B(x_0,r)  \subset \subset \EEE \Omega$ and all $GSBD$ competitors $v$ of $u$ in $B(x_0,r)$ we have
    \begin{align}
    \int_{B(x_0,r)} & \CC{e(u)} \dd{x} + M^{-1} \HH^{N-1}\big(J_u \cap B(x_0,r)\big) \nonumber \\ 
    &\leq \int_{B(x_0,r)} \CC{e(v)} \dd{x} + M \HH^{N-1}\big(J_v \cap B(x_0,r)\big) + h(r) r^{N-1}. \label{eq_GSBD_quasi}
    \end{align}
   
\item[2)] \label{defi_david_quasi}
    A coral pair $(u,K)$  with locally finite energy \EEE is a \emph{Griffith {local} $M$-quasiminimizer} with gauge $h$ in $\Omega$ if for any open ball $B(x_0,r)  \subset \subset \EEE \Omega$ and {all competitors} $(v,L)$ of $(u,K)$ in $B(x_0,r)$ we have
    \begin{align}
        \int_{B(x_0,r) \setminus K} & \CC{e(u)} \dd{x} + M^{-1} \HH^{N-1}\big(K \cap B(x_0,r)\big) \nonumber \\ 
        & \leq \int_{B(x_0,r) \setminus L} \CC{e(v)} \dd{x} + M \HH^{N-1}\big(L \cap B(x_0,r)\big) + h(r) r^{N-1}. \label{eq_david_quasi}
    \end{align}
    
    \item[3)] 
    {A coral pair $(u,K)$   with locally finite energy \EEE is a \emph{Griffith local topological $M$-quasiminimizer} with gauge $h$ in $\Omega$ if for any open ball $B(x_0,r)  \subset \subset \EEE \Omega$ and all topological competitors $(v,L)$ of $(u,K)$ in $B(x_0,r)$ inequality (\ref{eq_david_quasi}) holds.}
    \end{enumerate}
\end{definition}

For convenience, we will  often  just say `quasiminimizer' instead of `Griffith local  $M$-quasiminimizer'. As quasiminimal pairs are also topological quasiminimizers, our  study \EEE of regularity will focus on topological quasiminimizers. However, for our discussion in Remark \ref{rmk:bilip} below, it is convenient to {distinguish} this stronger notion of quasiminimal pair. \EEE In the above definition, the gauge plays the role of a normalized deviation from minimality.  Whereas in many works a certain decay is needed, i.e.,  gauges of the form  $h(r) = r^\alpha$ for $\alpha >0$ are considered, we allow for a wide class of gauges. In particular, we do not  require \EEE $\lim_{r \to 0} h(r) = 0$. \EEE The minimality property above is \emph{local} since conditions \eqref{eq_GSBD_quasi} and \eqref{eq_david_quasi} only hold in balls away from the boundary and it does not refer to an underlying Dirichlet condition on $\partial \Omega$.

Quasiminimizers satisfy an energetic upper bound. For example, given   a $GSBD$ $M$-quasiminimizer   $u$ \EEE with gauge $h$, for all open balls $B(x_0,r) \subset \Omega$, we have
\begin{equation}\label{eq_Gupperbound}
    \int_{B(x_0,r)} \CC{e(u)} \dd{x} + M^{-1} \HH^{N-1}(J_u \cap B(x_0,r)) \leq (M \sigma_{N-1} + {h(r)}) r^{N-1}.
\end{equation}
\BBB This is a straightforward application of (\ref{eq_GSBD_quasi}) with $v := u  \chi_{\Omega \setminus B(x_0,\rho)} \EEE$ for $\rho < r$. Precisely, $v$ is a competitor of $u$ in $B(x_0,t)$ for all $t \in (\rho,r)$. One applies the quasiminimality condition in $B(x_0,t)$ and then let $t \to \rho$, and next $\rho \to r$ to deduce the above bound. \EEE
For topological quasiminimizers, one has to check  additionally that the topological condition \eqref{eq_topological} is satisfied by $L := \left(K \setminus B(x_0,\rho)\right) \cup \partial B(x_0,\rho)$, but this is clear.

A primary result of this paper is the following theorem, which in conjunction with \eqref{eq_Gupperbound} shows that the crack set of a quasiminimizer is (locally) Ahlfors-regular, and in particular the surface energy is proportional to an $(N-1)$-dimensional surface at a positive and uniform scale.
\begin{theorem}\label{thm_af}
    For each $M \geq 1$, there exist constants $\varepsilon_{\rm Ahlf} \in (0,1)$ and $C_{\rm Ahlf} \geq 1$ (depending on $N$, $\C$, $M$) such that the following holds.
    \begin{enumerate}
        \item[1)] Let $u$ be a GSBD $M$-quasiminimizer in $\Omega$ with any gauge $h$ in the sense of Definition~\ref{defi_quasi}. Then, for all {$x_0 \in \Omega \cap \overline{J_u}$} and $r > 0$ such that $B(x_0,r) \subset \Omega$ and $h(r) \leq \varepsilon_{\rm Ahlf}$, we have
    \begin{equation*}
        \HH^{N-1}(J_u \cap B(x_0,r)) \geq C_{\rm Ahlf}^{-1} r^{N-1}.
    \end{equation*}
     In particular, if $\limsup_{r \to 0} h(r) < \varepsilon_{\rm Ahlf}$, \EEE the jump set is essentially closed,  i.e., 
     \begin{equation}\label{eqn:essClosed}
    \mathcal{H}^{N-1}\big((\overline{J_u}\setminus J_u) \cap \Omega \big) = 0.
    \end{equation}
    \item[2)]  Let $(u,K)$ be a topological $M$-quasiminimizer  in $\Omega$ \EEE with any gauge $h$ in the sense of Definition~\ref{defi_david_quasi}.
        Then, for all {$x_0 \in K$} and $r > 0$ such that $B(x_0,r) \subset \Omega$ and $h(r) \leq \varepsilon_{\rm Ahlf}$, we have
    \begin{equation*}
        \HH^{N-1}(K \cap B(x_0,r)) \geq C_{\rm Ahlf}^{-1}r^{N-1}.
    \end{equation*}
    \end{enumerate}
\end{theorem}

 The proof will be given in Section \ref{sec:density lower bound}. \EEE We now make a variety of remarks on the scope of   quasiminimality.

\begin{remark}[Bi-Lipschitz invariance] \label{rmk:bilip}
{\normalfont
Topological quasiminimizers are related to quasiminimal sets, introduced by {\sc David and Semmes} \cite[Definition 1.9]{DS}, which is a generalization of {\sc Almgren}'s `restricted sets' \cite[Definition II.1]{Almgren}. In essence, a set $K$ is quasiminimal if it satisfies
\begin{equation}\label{eq_quasi_set}
    {M^{-1}}\HH^{N-1}(K \cap B(x_0,r)) \leq {M}\HH^{N-1}(L \cap B(x_0,r))
\end{equation}
for all ball $B(x_0,r) \subset \subset \Omega$ and all topological competitors $(0,L)$ of \BBB $(0,K)$ \EEE in $B(x_0,r)$.
As the condition (\ref{eq_quasi_set}) is stable  under a bi-Lipschitz mapping $\phi$, in the sense that $\phi(K)$ is also a quasiminimal set (with a bigger constant), any regularity property that holds true for a quasiminimal set must be preserved under bi-Lipschitz mappings.  Consequently,  since the pair $(0,K)$ is a topological quasiminimizer as in Definition \ref{defi_david_quasi}, the regularity theory for Griffith {topological} quasiminimizers is   limited  in a similar fashion  and in general smooth crack sets cannot be expected.   We note that this limitation is not an artifact of the topological constraint and {also} applies {{to} the stronger notion of quasiminimizer} in the sense of {Definition \ref{defi_quasi}, 2)}. Precisely, one can show that quasiminimizers of the Mumford--Shah functional remain quasiminimal  under small bi-Lipschitz perturbations, \EEE   and then identify a quasiminimizer of the Mumford--Shah functional with a Griffith quasiminimizer having only antiplanar shear.
}
\end{remark}

In the case that $M=1$ and {$\lim_{r \to 0} h(r) = 0$}, a quasiminimizer is called an \emph{almost-minimizer}  of the Griffith energy. {If in addition} $h\equiv 0$, we use the term  \emph{local minimizer}. \EEE
{In contrast to Remark \ref{rmk:bilip}, we expect that the crack of an almost-minimizer is much more than just Ahlfors-regular, as one inherits properties associated with minimality at sufficently small scales.} 
We refer to \cite{BIL,FLS,CL} for  finer \EEE regularity properties of (almost-)minimizers of the Griffith energy.
 We emphasize \EEE that we follow the terminology of {\sc David} \cite{David}, and our almost-minimizers with gauge $h(r) = C r^{\alpha}$ correspond to what is called quasiminimizers in \cite[Definition 7.17]{AFP} and \cite[Definition~5.1]{CC}.

\begin{remark}[Comparison to other notions of quasiminimality]\label{rmk:otherQuasi}{\normalfont
 (a)   We remark that Definition \ref{defi_david_quasi} is more general than the definition of quasiminimizers introduced by {\sc David} \cite[Definition 7.21]{David}.
    Translated  to the \EEE Griffith setting, the quasiminimality condition \cite[(7.20)]{David} for the pair $(u,K)$ says that, for any  competitor $(v,L)$,
    \begin{equation}\label{eq_david_quasi0}
        \HH^{N-1}(K \setminus L) \leq M \HH^{N-1}(L \setminus K) + \Delta E + a r^{N-1},
    \end{equation}
    where
    \begin{equation*}
        \begin{gathered}
            \Delta E  := \EEE \max \big\{E(v) - E(u), M \big(E(v) - E(u)\big) \big\} \text{  and } \EEE \\
            E(v) = \int_{  B(x_0,r) \EEE  \setminus L} \CC{e(v)} \dd{x}, \quad E(u) = \int_{  B(x_0,r) \EEE \setminus K} \CC{e(u)} \dd{x}.
        \end{gathered}
    \end{equation*}
     In the definition of $\Delta E$, note that $E(v) - E(u)$ can be positive or negative. \EEE 
    A direct computation shows that a quasiminimizer in this sense is a quasiminimizer in the sense of Definition \ref{defi_david_quasi}.
    {\sc David}'s quasiminimality condition \eqref{eq_david_quasi0} is natural in the context of the classical Griffith (or Mumford--Shah) energy, where the surface energy density of a crack is a constant, as it does   not take into account  the contribution of  $K \cap L$ present in both pairs $(u,K)$ and $(v,L)$. \EEE   On the other hand, our notion of quasiminimality in Definition \ref{defi_david_quasi} applies to local minimizers for a wide class of functionals, as in Theorem~\ref{cor:aniso} below. 
    
(b)   In \cite[Definition 2.1]{BL}, {\sc Bucur and Luckhaus} introduce the closely related notion of a ($\Lambda$, $\alpha$, $c_\alpha$) almost quasiminimizer for the Mumford--Shah functional. Though our definition is slightly more general  concerning the class of admissible gauges,  one sees that if $u$ is an $M$-quasiminimizer in our sense (adapted to the Mumford--Shah setting) with gauge $h(r):= {h(1)} r^\alpha$, $\alpha \in (0,1)$, then $\sqrt{M}u$ is an ($M^2$, $\alpha$, {$M h(1)$}) almost quasiminimizer.} 
\end{remark}

We finally discuss the difference between the different notions of quasiminimality given  in  Theorem~\ref{thm_af}.

\begin{remark}[Topological versus $GSBD$ quasiminimality]
{\normalfont
According to Theorem \ref{thm_af}, if $u$ is a $GSBD$ quasiminimizer with a small gauge $h$, then the pair $(u,\overline{J_u})$ is a quasiminimizer in the classical sense.
Reciprocally, for a quasiminimizer $(u,K)$,  one can check that $u \in GSBD^2_{\mathrm{loc}} (\Omega) \EEE $ and that {$K  = \overline{J_u}$}  (with strict inclusion $ \overline{J_u} \subset K$, a small ball could be removed from $K$ contradicting minimality). However \EEE when translating the minimality condition \eqref{eq_david_quasi} to the $GSBD$ setting, it compares the closure of the jump sets. Although this looks probable, it is not clear how to prove that $u$ is a $GSBD$ quasiminimizer as this  would require \EEE to compare  the jump sets and not their closure.
In the case where $(u,K)$ is {only} a topological quasiminimizer, $\overline{J_u}$ may even be a strict subset of $K$,  see \EEE  the example $(0,K)$ in Remark \ref{rmk:bilip}.
}
\end{remark}

 \subsection{Application to models in linear elasticity with surface discontinuities}

One of the strengths of Definition \ref{defi_quasi} is that it includes local minimizers of a large class of functionals. In particular, we show that $GSBD$ minimizers of generic fracture energies are in fact strong minimizers. We phrase our result in terms of the near minimal assumptions on the bulk and surface energy densities, and  briefly discuss in Remark \ref{rmk:...manuel}   specific examples where minimizers of the weak formulation have been shown to exist.

\begin{theorem}[`Weak-to-strong' regularity of local \EEE minimizers]\label{cor:aniso}
Let $u \in GSBD^2  (\Omega)  $ be a (local) \EEE minimizer of the functional
\begin{equation}\label{eqn:anisoEnergy}
        \int_{\Omega} f(x,e(u)) \dd{x} + \int_{J_u} \psi(x,u^+,u^-,\nu_u) \dd{\HH^{N-1}},
    \end{equation}
    where $\psi \colon \Omega \times \R^N \times \R^N \times  \mathbb{S}^{N-1} \EEE \to \R_+$ and $f \colon \Omega \times \R^{N \times N}_{\mathrm{sym}} \to \R_+$ are Borel functions satisfying $M^{-1} \leq \psi \leq M$,  $\psi(x,a,b,\nu) = \psi(x,b,a,  - \EEE \nu)$ for all $(x,a,b, \nu) \in \Omega \times \R^N \times \R^N \times  \mathbb{S}^{N-1} $,   and
    \BBB
    \begin{equation}\label{eq_gxi}
    \abs{f(x,\xi) - \CC{\xi}} \leq g(\abs{\xi}^2) \quad \text{ for all } x\in \Omega \text{ and } \xi \in \R_{\mathrm{sym}}^{N \times N},
    \end{equation}
where $g \colon [0,+\infty) \to [0,+\infty)$ is a nondecreasing function such that $\int_{1}^{+\infty} s^{-2} g(s) \dd{s} < +\infty$. \EEE
    Then $J_u$ is Ahlfors-regular (\ref{eqn:AhlfIntro}) and essentially closed \eqref{eqn:essClosed}.  In particular, $(u,\overline{J_u})$ is a (local) \EEE minimizer of the strong formulation associated with (\ref{eqn:anisoEnergy}).
\end{theorem}
\BBB A particular case of interest for (\ref{eq_gxi}) are perturbations of type 
\begin{align}\label{eq_gxi2}
g(\abs{\xi}^2) = M(1 + \abs{\xi}^q),
\end{align} for some constant $M \geq 1$ and exponent $q \in (0,2)$, cf.\ also \cite{new-ref2}. \EEE {This is the case for example for} densities of the form $f(\xi) = \abs{\xi - S}^2$, where $S:\Omega \to \R^{N\times N}_{\rm sym}$ is a field of bounded symmetric matrices, arising in the dual formulation of the biharmonic optimal compliance problem \cite{LemPakzad}. 
    The general surface term can take into account material anisotropy, (non-degenerate) cohesive material  properties, \EEE or more complex behaviors across cracks,  such  \EEE as in the thermal insulation problem studied in \cite[Section 4]{BL} (in the scalar case). 
    {Note that the second assumption on $\psi$ is needed because the tuple $(x,u^+,u^-,\nu_u)$ is only determined up to a change of sign of $\nu_u$ and an interchange of $u^+$ and $u^-$.}

 \begin{remark}[Existence]\label{rmk:...manuel}   
 {\normalfont
 
     Suppose that $\Omega$ is a bounded Lipschitz domain and let $\Omega' \supset \Omega$ be a bounded domain such that $\partial_D\Omega = \partial \Omega \cap \Omega'$ is relatively open. We consider   $u_0 \in W^{1,2}(\Omega';\R^N)$ representing Dirichlet boundary values on $\partial_D\Omega$.  We suppose that, besides the assumptions in   Theorem \ref{cor:aniso}, we have that {$\psi = \psi(x,\nu)$} is continuous in $x$ and   independent of $u^+$ and $u^-$, and that  $f \colon \Omega \times \R^{N \times N}_{\mathrm{sym}} \to \R_+$ is  a Carathéodory function which  is symmetric quasiconvex for a.e.\ $x$, i.e., 
 $${f(x,\xi)} \le \frac{1}{\mathcal{L}^N(A)} \int_A f(x,    \xi  + e(\varphi)(y)   ) \, {\rm d} y   \quad \text{ for all $x \in \R^{N \times N}_{\rm sym}$ and   $\varphi \in  W^{1,\infty}_0 \EEE (A;\R^N)$}$$
 for every bounded open set $A \subset \R^N$. In this setting, existence of (weak) minimizers $u \in GSBD^2(\Omega)$ of \eqref{eqn:anisoEnergy}  (on $\Omega'$ in place of $\Omega$) \EEE under the Dirichlet condition $u = u_0$ on $\Omega' \setminus \overline{\Omega}$ is guaranteed, see \cite[Proposition~4.4]{CC2020}.
 
The case where $\psi$ additionally depends on the traces $u^+$ and $u^-$ is more delicate. Here, existence can be shown for a relaxed problem   \cite[Proposition 4.5]{CC2020}, \EEE under additional properties for $\psi$ given in  \cite[Theorem 1.2]{CC2020}.  \EEE In particular, an assumption on lower semicontinuity of surface integrals is needed  \cite[$(g_4)$]{CC2020} which is {for instance} satisfied if $\psi$ is symmetric jointly convex, \BBB see    \cite[Theorem 5.1]{perugini}.    \EEE
\BBB In specific situations, the approach of Theorem \ref{cor:aniso} can also be applied to functionals where $\psi$ is not bounded from below. For instance, in the thermal insulation problem \cite{BL}, $\psi$ is given by $(u^+)^2 + (u^-)^2$ but the authors establish that, if $u$ is a minimizer, this density is actually bounded from below along $J_u$ by a universal constant. Therefore, the minimizers of this problem also fall within the class of quasiminimizers.

 }
 \end{remark}

 \begin{remark}[Bulk energy]   
 {\normalfont
    The main limitation to the choice of $f$ is that we rely on regularity estimates for local minimizers of $$v \mapsto \int_{B(0,1)} \CC{e(v)} \dd{x}.$$ 
    In particular, Theorem \ref{cor:aniso} holds provided that \eqref{eq_gxi} is replaced by
    \begin{equation*}
        \BBB \abs{f(x,\xi) - f_0(\xi)} \leq g(\abs{\xi}^p) \quad \text{ for all }x\in \Omega \text{ and } \xi \in \R^{N \times N}_{\mathrm{sym}},
    \end{equation*}
    where $f_0$ is a convex $p$-homogeneous function, with $$\frac{1}{C_0}|\xi|^p \leq f_0(\xi) \leq C_0|\xi|^p \quad \text{ for all } \xi \in \R^{N \times N}_{\mathrm{sym}}$$ for some $C_0>0$, such that any local minimizer $u$ of the energy $\int_{B(0,1)} f_0(e(v)) \dd{x}$ satisfies the decay estimate
    $$\int_{B(0,r)} f_0(e(u)) \dd{x} \leq C r^{N-1+\alpha} \int_{B(0,1)} f_0(e(u)) \dd{x} \quad \text{ for all }r\in (0,1)$$
    for some constants $C>0$ and $\alpha>0$ depending only on $f_0$ and the dimension.
    \EEE
    In particular, by \cite[Proposition 3.4 and  Corollary \EEE 4.3]{CFI0}, in dimensions $N=2$ and $3$, our results apply when $f_0(\xi) = {(\C\xi:\xi)^{p/2}}$ for $p \in (1,+\infty)$.  (In fact, as in \cite{CCI}, the approximation result needed in the proof of Theorem~\ref{thm_af}, see Section~\ref{sec:Approx}, carries over to general $p$, which we do not perform in this paper for notational convenience.  Accordingly, in the decay estimate Proposition \ref{prop_decay} one needs to replace $\tau^{N}$ by $\tau^{N-1+\alpha}$ which is enough to prove Theorem \ref{thm_af}.) \EEE
    Similarly, our results {{adapt} to} the fully-elliptic setting when {$f(x,e(u))$ is replaced by $f(x,\nabla u)$}. {In this case,} Theorem \ref{cor:aniso} gives existence of strong minimizers to the relaxation of the (nonlinearly) frame indifferent energy
    $$ \int_{\Omega} {\rm dist}^2(\nabla u, {\rm SO}(N)) \dd{x} + \mathcal{H}^{N-1}(J_u),$$
    defined for $u$ in $GSBV^2(\Omega;\R^{N})$, or for energies with more general surfaces densities as considered in~\eqref{eqn:anisoEnergy}. 
   }
\end{remark}

\begin{proof}[Proof of Theorem \ref{cor:aniso}]
 We prove that a local minimizer $u \in GSBD^2  (\Omega) \EEE $ of the functional \eqref{eqn:anisoEnergy} is a $GSBD$ $M$-quasiminimizer for a suitable gauge $h$ and $M \ge 1$ \EEE in the sense of Definition \ref{defi_quasi}. With this, Theorem~\ref{thm_af}  concludes \EEE  the result. \BBB First, let us notice that, since $g$ is nondecreasing,  for all $t > 0$ we have
    \begin{equation*}
        \frac{g(t)}{t} \leq \int_t^{+\infty} \frac{g(s)}{s^2} \dd{s},
    \end{equation*}
    which converges to zero as $t \to +\infty$.
It follows that there exists $M_1 \geq 1$ (depending on $\mathbb{C}$ and $g$) such that for all $\xi \in \R^{N \times N}_{\mathrm{sym}}$ we have $g(\abs{\xi}^2) \leq (\CC{\xi})/2 + M_1$, and thus by \eqref{eq_gxi}
    \begin{equation}\label{eq_gxi3}
        \CC{\xi} \leq 2f(x,\xi) + 2M_1 \quad \text{ for all }x\in \Omega \text{ and } \xi \in \R^{N \times N}_{\mathrm{sym}}.
    \end{equation} \EEE
    We fix an open ball $B(x_0,r) \subset \Omega$.
    \BBB  Setting $h(r) = +\infty$ if $r>1$, we can assume without restriction that  $r \leq 1$. \EEE
    By comparing $u$ with  the competitor   $u   \chi_{\Omega\setminus B(x_0,\rho)} \EEE $ for $0 < \rho < r$, using ${f(x,0)} \leq \BBB g(0)\EEE$ and $M^{-1} \leq \psi \leq M$,   and sending $\rho \to r$, the local minimality of $u$ yields \EEE
    \begin{equation*}
        \int_{B(x_0,r)} f(x,e(u)) \dd{x} + M^{-1} \HH^{N-1}(J_u \cap B(x_0,r)) \leq \BBB g(0) \EEE \, \omega_N r^N + M \sigma_{N-1} r^{N-1},
    \end{equation*}
    where $\omega_N$ is the volume of the unit ball and $\sigma_{N-1}$ is the surface area of the unit sphere.     Then,  \BBB as $r \le 1$, \EEE from \eqref{eq_gxi3}  it follows \EEE that
    \begin{equation}\label{eq_gupper}
        \int_{B(x_0,r)} \CC{e(u)} \dd{x} + M^{-1} \HH^{N-1}(J_u \cap B(x_0,r)) \leq  C_0 r^{N-1}
    \end{equation}
    for some constant $C_0 \geq 1$ which depends on $N$, $M$,  $M_1$, \BBB and $g(0)$. \EEE
    Now, we let $v$ be a  {$GSBD$} \EEE competitor of $u$ in $B(x_0,r)$  in the sense of Definition \ref{def:com},1). \EEE  By minimality of $u$, we have
    \begin{align}
        \int_{B(x_0,r)} &  f(x,e(u)) \dd{x}  + \int_{J_u \cap B(x_0,r)} \psi(x,u^+,u^-,\nu_u) \dd{\HH^{N-1}}\nonumber \\ 
        &\leq \int_{B(x_0,r)} f(x,e(v)) \dd{x} + \int_{J_v \cap B(x_0,r)} \psi(x,v^+,v^-,\nu_v) \dd{\HH^{N-1}}. \label{eqn:starmin}
    \end{align}
 We  now \EEE   show that this implies the quasiminimality condition \eqref{eq_GSBD_quasi}     for some gauge $h$ to be  specified.    In fact, \EEE if
    \begin{equation*}
        \int_{B(x_0,r)} \CC{e(v)} \dd{x} > C_0 r^{N-1},
    \end{equation*}
    inequality \eqref{eq_GSBD_quasi} holds directly for $h \equiv 0$ \EEE   {by \eqref{eq_gupper}.}
    Otherwise, \BBB by \eqref{eq_gxi} \EEE we have  
    \begin{equation*}
        \int_{B(x_0,r)} f(x,e(v)) \dd{x} \leq \int_{B(x_0,r)} \CC{e(v)} \dd{x} + \BBB \int_{B(x_0,r)} g(\abs{e(v)}^2) \dd{x}. \EEE
    \end{equation*} \BBB
 Our goal is to show that
    \begin{equation}\label{eq_gcontrol}
        \int_{B(x_0,r)} g(\abs{e(v)}^2) \dd{x} \leq h(r) r^{N-1},
        \quad \text{where} \quad
        h(r) := C_1 C_0 \int_{1/r}^{+\infty} \frac{g(s)}{s^2} \dd{s},
        \end{equation}
    for some $C_1 \geq 1$ which depends only on $N$.
    For this purpose, we estimate the integral separately on $B(x_0,r) \cap \set{\abs{e(v)}^2 \leq 1/r}$ and $B(x_0,r) \cap \set{\abs{e(v)}^2 > 1/r}$. For the first part, using the fact that $g$ is nondecreasing, we estimate
    \begin{equation*}
        \int_{B(x_0,r) \cap \set{\abs{e(v)}^2 \leq 1/r}} g(\abs{e(v)}^2) \dd{x} \leq \omega_N r^N g(1/r) \leq \omega_N  r^{N-1}   \int_{1/r}^{+\infty} \frac{g(s)}{s^2} \dd{s}  .
    \end{equation*}
    In order to deal with the second part, for $t > 0$ we define 
     $$g^{-1}(t) := \mathrm{inf} \set{s \in (0,+\infty) | g(s) \geq t}.$$
 \BBB (Clearly, if $g$ is strictly increasing, $g^{-1}$ is the inverse.) \BBB If $g(s) \geq t$, we have $s \geq g^{-1}(t)$, and  if $s > g^{-1}(t)$, then $g(s) \geq t$. Now, we use the layer cake representation and the assumption $\int_{B(x_0,r)} \CC{e(v)} \dd{x} \leq C_0 r^{N-1}$ to bound
    \begin{align*} 
        \int_{B(x_0,r) \cap \lbrace \abs{e(v)}^2 > 1/r \rbrace} g(\abs{e(v)}^2) \dd{x} &= \int_{0}^{+\infty} \big|\set{x \in B(x_0,r) \colon \,  \abs{e(v)}^2 \geq \max(1/r, g^{-1}(t))} \big| \dd{t}\\
                                                                                   &\leq \int_{0}^{+\infty} \frac{1}{\max(1/r,g^{-1}(t))} \left(\int_{B(x_0,r)} \abs{e(v)}^2 \dd{x}\right) \dd{t}\\
                                                                                   &\leq C C_0 r^{N-1} \int_{0}^{+\infty} \min(r,1/g^{-1}(t)) \dd{t}.
    \end{align*} 
    Using the layer cake representation again and Fubini's theorem, we find
    \begin{align*}
        \int_{0}^{\infty} \min(r,1/g^{-1}(t)) \dd{t} &= \int_0^{\infty} \int_0^{r} \chi_{\set{0 < s < 1/g^{-1}(t)}} \dd{s} \dd{t}
                                                             \leq \int_0^{\infty} \int_0^{r} \chi_{\lbrace s \colon  g(1/s) \geq t\rbrace } \dd{s} \dd{t}\\
                                                             &\leq \int_0^{r} g(1/s) \dd{s} = \int_{1/r}^{+\infty} \frac{g(s)}{s^2} \dd{s}.
    \end{align*}
This proves (\ref{eq_gcontrol}).
    \EEE 
Similarly, by \eqref{eq_gupper},  \BBB   \eqref{eq_gxi},  and \eqref{eq_gcontrol} \EEE we have
    \begin{align*}
        \int_{B(x_0,r)} f(x,e(u)) \dd{x} &\geq \int_{B(x_0,r)} \CC{e(u)} \dd{x} -  \BBB \int_{B(x_0,r)} g(\abs{e(u)}^2) \dd{x}\\
                                         &\geq \int_{B(x_0,r)} \CC{e(u)} \dd{x} - h(r) r^{N-1}.
\end{align*}
\EEE
We may plug the above estimates into (\ref{eqn:starmin}) and use the bound $M^{-1}\leq \psi \leq M$ to see that   $u$ is a $M$-quasiminimizer in the sense of \eqref{eq_GSBD_quasi} with a gauge $h$ defined \BBB for all $r \leq 1$ as in \eqref{eq_gcontrol}. \BBB As $s^{-2} g(s)$ is integrable on $[1,+\infty)$, we have $\lim_{r \to 0} h(r) = 0$ and therefore Theorem \ref{thm_af} applies. \EEE
   \end{proof}

\section{Density lower bound for quasiminimizers}\label{sec:density lower bound}

In this section   we \EEE prove Theorem \ref{thm_af}.
{The heart of the proof is a decay estimate} for the localized Griffith energy in a region only having a small crack. With this decay estimate at hand, we may directly conclude Theorem \ref{thm_af}. Our proof of the decay estimate requires a study of limits of quasiminimizers with vanishing crack, analogous to \cite{CFI,CCI}. For the case of $GSBD$ quasiminimizers, we will rely on an approximation result proven in \cite{CCI}. For topological quasiminimizers  instead, \EEE we will need to ensure that this approximation, found using minimality, respects both the topological constraints for competitors and that the crack is closed. Given the similarity between the two cases, we will focus our attention on the more involved case of topological quasiminimizers and highlight places in the proof where $K$ cannot  be replaced simply by $J_u$.  For convenience, in place of integrals of the form $\int_{B(x_0,r) \setminus K}$, we always write  $\int_{B(x_0,r)}$ which clearly does not affect the integral as $K$ is a Lebesgue-negligible set. \EEE

\subsection{Decay property of the Griffith energy}

We adopt the approach used in \cite{CFI} and \cite{CCI}, which was motivated by the classical result of {\sc De Giorgi et al.} \cite{DCL} (also presented in \cite{AFP}), but also introduced new compactness results for the vectorial Griffith setting.
Given an open ball $B(x_0,r)$ and a pair $(u,K)$, we let $G(x_0,r)$ denote the Griffith energy of $(u,K)$ in $B(x_0,r)$,  namely \EEE
\begin{equation}\label{eqn:GriffithEnergyLoc}
    G(x_0,r) := \int_{B(x_0,r)} \CC{e(u)} \dd{x} + \HH^{N-1}\big(K \cap B(x_0,r)\big).
\end{equation}
We show  the \EEE following decay estimate   for the Griffith energy of a quasiminimizer.
\begin{proposition}[Decay estimate]\label{prop_decay}
    For each $M \geq 1$ and $\tau \in (0,1)$, there exists $\varepsilon \in (0,1)$ (depending on $N$, $\C$, $M$, $\tau$) such that the following holds.
    If $u$ is a topological (or $GSBD$) $M$-quasiminimizer with gauge $h$ in a ball {$B(x_0,r) \subset \Omega$} and if
    \begin{equation*}
         \frac{\HH^{N-1}(  K \EEE \cap B(x_0,r))}{r^{N-1}} \leq \varepsilon \quad \text{and} \quad h(r) \leq \varepsilon  \frac{G(x_0,r)}{r^{N-1}},
    \end{equation*}
    then we have
    \begin{equation*}
        G(x_0,\tau r) \leq C \tau^{N} G(x_0,r),
    \end{equation*}
    for some constant $C \geq 1$ which depends  only \EEE on $N$ and $\C$.
\end{proposition}

The essence of the above proposition is that, if the relative amount of crack in a ball is small, then the localized Griffith energy inherits the decay given by the energy with no crack. In particular, the decay scales like that for minimizers of the energy $\int \CC{e(u)}\dd{x}$, which solve a standard linear elliptic PDE.   We first prove Theorem \ref{thm_af} in Subsection \ref{sec: main proof} and afterwards we give the proof of the decay estimate in Subsection \ref{sec: decay}. \EEE

\subsection{Proof of Theorem \ref{thm_af}}\label{sec: main proof}
Below we prove (2) of Theorem \ref{thm_af}, and comment on the slight change for $GSBD$ quasiminimizers in the remark following the proof.

\begin{proof}[Proof of Theorem \ref{thm_af}]
We proceed in three steps. In Step 1, we show that  one can reduce the problem  to considering \EEE points in the crack  $K$ \EEE with their $(N-1)$-dimensional density uniformly bounded away from $0$. In Step 2, we find an energy lower bound for the localized Griffith energy. For this, we argue by contradiction and use the decay estimate in Proposition \ref{prop_decay} to show that, if the Griffith energy is small to start with, the Griffith energy stays small at all smaller scales.  This in turn yields  small $(N-1)$-dimensional density of  the surface energy, contradicting  the density bound found in Step 1. \EEE    Finally, in Step 3, we conclude  using a similar argument by contradiction. \EEE

\textit{Step 1 (Uniform lower bound on $(N-1)$-dimensional density).}
We introduce the set
\begin{equation*}
    K^* := \Set{x \in K  \colon \EEE \, \limsup_{r \to 0} \frac{\HH^{N-1}\big(K \cap B(x,r)\big)}{\omega_{N-1} r^{N-1}} \geq 2^{-(N-1)}},
\end{equation*}
where $\omega_{N-1}$ is the measure of the  $(N-1)$-dimensional  unit \EEE disk. As $K$ is $\HH^{N-1}$-locally finite (but not necessarily rectifiable) by \cite[(2.43)]{AFP} for $\HH^{N-1}$-a.e. $x \in K$ it holds that
\begin{equation*}
    \limsup_{r \to 0} \frac{\HH^{N-1}\big(K \cap B(x,r)\big)}{\omega_{N-1} r^{N-1}} \geq 2^{-(N-1)},
\end{equation*}
and in particular, $\HH^{N-1}(K \setminus K^*) = 0$.
We further claim that
\begin{equation}\label{eqn:subsetrel}
    \overline{K^*} \cap \Omega = K.
\end{equation}
The inclusion $\overline{K^*} \cap \Omega \subset K$ is clear as $K^*\subset K$  and $K$ is closed. \EEE To see that $K \subset \overline{K^*}$,   suppose by contradiction that there exists \EEE $x\in K \setminus \overline{K^*}$.

Then there is a radius $r > 0$ such that $K^* \cap B(x,r) = \emptyset$ and, since $K^*$ contains $\mathcal{H}^{N-1}$-a.e.\ point in $K$, we must have
\begin{equation}\label{eqn:densityContra}
\mathcal{H}^{N-1}\big(K\cap B(x,r)\big) = 0.
\end{equation}
\EEE
But as $K$ is coral, this {contradicts} the assumption $x\in K$, {see (\ref{eq_coral})}. This concludes the proof of (\ref{eqn:subsetrel}).  Due to \eqref{eqn:subsetrel} and a density argument, in the next steps we can focus on points in $K^*$. \EEE

\newcommand{\Co}{C_0}
\newcommand{\varepso}{\varepsilon_0}
\newcommand{\tauo}{\tau_0}

\textit{Step 2 (Lower bound on Griffith energy at uniformly positive scale).}
We show that we can choose constants $\varepso \in (0,1)$ and $\Co \geq 1$ (depending only on $N$, $\C$, $M$) such that for all {$x \in K$} and  $r > 0$ such that {$B(x,r) \subset \Omega$} and $h(r) \leq {\varepso^3}$,   we have
    \begin{equation}\label{eq_G}
        G(x,r) \geq \Co^{-1} r^{N-1}.
    \end{equation}
    We fix  a \EEE  parameter $\tauo \in (0,1)$ to be chosen later (depending only on $N$ and $\C$, see right after \eqref{eqn:decayCaseOne}). With this, we fix $\varepso$  with $0< \varepso < \min\lbrace \tauo^{N-1},\omega_{N-1} (\tauo/2)^{N-1}\rbrace$ \EEE  such that statement of Proposition \ref{prop_decay} holds {with respect to $\tauo$.} Finally, we note that  {by  \eqref{eqn:subsetrel}} it is sufficient to prove (\ref{eq_G}) for points {$x \in K^*$}.

    \emph{{Substep} 2.1 (One step).} Let $x \in K^*$. We claim that, whenever {$B(x,r) \subset \Omega$} with
    \begin{equation}\label{eq_G1}
        G(x,r) \leq \varepso r^{N-1} \quad \text{ and } \quad  h(r) \leq \varepso^2 \tauo^{N-1},
    \end{equation}
    then
    \begin{equation}\label{eqn:des_est}
        G(x, \tauo \EEE  r) \leq \varepso (\tauo r)^{N-1}.
    \end{equation}
    Indeed,   condition \eqref{eq_G1} implies $\HH^{N-1}(K \cap B(x,r)) \leq \varepso r^{N-1}$, and then two possibilites can occur.
    If we have $h(r) \leq \varepso r^{-(N-1)} G(x,r)$, Proposition \ref{prop_decay} and (\ref{eq_G1}) show that 
    \begin{equation}\label{eqn:decayCaseOne}
        G(x,\tauo r) \leq C \tauo^N G(x,r) \leq C \varepso \tauo^N r^{N-1}
    \end{equation}
    for some constant $C \geq 1$ that depends on $N$ and $\C$.   Taking $\tauo$ small enough such that $C \tauo \leq 1$, inequality (\ref{eqn:decayCaseOne}) gives the desired estimate (\ref{eqn:des_est}).
    If on the other hand $h(r)  > \EEE \varepso r^{-(N-1)} G(x,r)$, then we trivially have
    \begin{equation*}
        G(x,\tauo r) \leq G(x,r) \leq h(r) \varepso^{-1} r^{N-1},
    \end{equation*}
    and the assumption $h(r) \leq \varepso^2 \tauo^{N-1}$ yields the estimate (\ref{eqn:des_est}).

    \emph{{Substep} 2.2 (Iteration).} We deduce by iteration that, if there exists some $r > 0$ such that {$B(x,r) \subset \Omega$}, $h(r) \leq \varepso^{2} \tauo^{N-1}$, and $G(x,r) \leq \varepso r^{N-1}$, then for all integers $k \geq 0$,
    \begin{equation*}
        G(x, \tauo^k r) \leq \varepso (\tauo^k r)^{N-1}.
    \end{equation*}
 This follows from {Substep} 2.1 and the fact that $h$ is nondecreasing. \EEE     For all $0 < \rho \leq r$, there exists an integer $k \geq 0$ such that $\tauo^{k+1} r \leq \rho \leq \tauo^k r$, and thus
    \begin{equation*}
        G(x,\rho) \leq G(x,\tauo^k r) \leq \varepso \EEE (\rho/\tauo)^{N-1}.
    \end{equation*}
    In turn, this implies that
    \begin{equation*}
        \limsup_{\rho \to 0} \frac{\HH^{N-1}\big(K \cap B(x,\rho)\big)}{\rho^{N-1}} \leq \varepso \tauo^{-(N-1)} < \omega_{N-1}2^{-(N-1)},
    \end{equation*}
   where we used $\varepso < \omega_{N-1} (\tauo/2)^{N-1}$. This \EEE   contradicts the assumption $x \in K^*$.
   We conclude that for all $x \in K^*$ and all $r > 0$ such that {$B(x,r) \subset \Omega$} and {$h(r) \leq \varepso^{3}$ (note $\varepso^{3} \leq \varepso^{2} \tauo^{N-1}$)}, (\ref{eq_G}) holds with $\Co = 1/\varepso$.

\textit{Step 3 (Lower bound on surface energy at uniformly positive scale).}
    We take $\tau \in (0,1)$ as a small constant to be chosen below (depending on $N$, $\C$ and $M$) and let $\varepsilon > 0$ be the constant associated with $\tau$ coming from Proposition \ref{prop_decay} (which depends on $N$, $\C$, $M$).

    If there exists $x \in K$ and $r > 0$ with {$B(x,r) \subset \Omega$} such that $h(r) \leq \min\lbrace \Co^{-1} \varepsilon,  \varepso^3\rbrace \EEE $ and
    \begin{equation}\label{eqn:finalContra}
        \HH^{N-1}\big(K \cap B(x,r)\big) \leq \varepsilon r^{N-1},
    \end{equation}
    then by \eqref{eq_G} we have $h(r) \leq \varepsilon r^{-(N-1)} G(x,r)$, and thus Proposition \ref{prop_decay} shows that
    \begin{equation*}
        G(x,\tau r) \leq C \tau^N G(x,r).
    \end{equation*}
    By \eqref{eq_Gupperbound}, we have $G(x,r) \leq C r^{N-1}$ (for a constant $C>0$ depending on $M$), so
    \begin{equation*}
        G(x,\tau r) \leq C \tau^N r^{N-1}
    \end{equation*}
     for $C>0$ depending on $M$,  $N$, and $\C$. \EEE     If $\tau$ is taken small enough (depending on $N$, $\C$ and $M$) this contradicts \eqref{eq_G}. Consequently, \eqref{eqn:finalContra} cannot hold for any $r>0$ with  $B(x,r)\subset \subset \Omega $ and \EEE $h(r) \leq \min\lbrace \Co^{-1} \varepsilon, \varepso^3\rbrace$\EEE. Thus, the statement holds for ${C_{\rm Ahlf}} := \eps^{-1}$ and $\eps_{\rm Ahlf}:=\min\lbrace \Co^{-1} \varepsilon,  \varepso^3\rbrace \EEE $.
\end{proof}

\begin{remark}[The $GSBD$ case]
{\normalfont
    Proceeding as in Step 1 of the above proof for $GSBD$ quasiminimizers, with $K$ replaced by $J_u$ (or occasionally $\overline{J_u}$), one must show $  \overline{J_u^*} \cap \Omega = \overline{J_u} \cap \Omega$, the analogue of \eqref{eqn:subsetrel}. Herein, (\ref{eqn:densityContra}) becomes 
\begin{equation}\label{eqn:densityContra2}
    \HH^{N-1}\big(J_u \cap B(x, r)\big) \EEE = 0 
\end{equation}
for some $ x \EEE \in J_u \setminus \overline{J_u^*}$ {and $r > 0$ small enough such that $B(x,r)\subset \Omega$ and $h(r) < \infty$}. To see that $ x \EEE$ cannot belong to $J_u$, note that by the definition of $GSBD$ via slices (see \cite{DalMaso:13}) and (\ref{eqn:densityContra2}) the distributional symmetric derivative is absolutely continuous, and by Korn's inequality, one finds that $u \in W^{1,2}(B( x, \EEE r);\R^N)$.
But this contradicts the assumption $ x \EEE \in J_u$ thanks to Proposition \ref{prop_campanato} in    Appendix  \ref{sec:appendix}  (note that \eqref{eq_campanato} holds by \eqref{eq_Gupperbound}). \EEE Steps 2 and 3 remain unchanged, replacing $K$ by $J_u$.  Provided that $\limsup_{r \to 0} h(r) < \varepsilon_{\rm Ahlf}$, we eventually get \eqref{eqn:essClosed} as a standard consequence of the density lower bound along $\overline{J_u}$, see \cite[(2.42)]{AFP}. \EEE
} 
\end{remark}

\subsection{Proof of the decay estimate}\label{sec: decay}
The proof of the decay estimate, i.e., Proposition \ref{prop_decay}, is achieved by a contradiction and compactness argument.
For this, we will need to understand the limit behavior of quasiminimizing sequences with asymptotically vanishing crack.

 We start with explaining the proof idea. Working on the unit ball and supposing that the conclusion of Proposition \ref{prop_decay} fails, we   have a sequence of pairs $(u_n,K_n)$  with vanishing crack, in the sense that $\lim_{n \to   \infty} \HH^{N-1}(K_n) =0$, and such that 
\begin{equation}\label{eqn:contraHyp}
G_n(0,\tau)>C\tau^NG_n(0,1) ,
\end{equation}
where $\tau \in (0,1)$ is fixed and $G_n$ is the Griffith energy of $(u_n,K_n)$ as in (\ref{eqn:GriffithEnergyLoc}) and the constant $C>0$ will be chosen  appropriately. \EEE   To derive a contradiction, we will show that $u_n$ converges to $u \in W^{1,2}(B(0,1);\mathbb{R}^N)$ minimizing
$\int_{B(0,1)} \CC{e(u)}\dd{x}$ subject to its own boundary conditions. As $u$ solves a  linear \EEE elliptic PDE, we have in particular that
\begin{equation}\label{eqn:contraDecay}
\int_{B(0,\tau)} \CC{e(u)}\dd{x} \leq C_0 \tau^N \int_{B(0,1)} \CC{e(u)}\dd{x}
\end{equation} 
for a universal constant $C_0>0.$ We will pass to the limit in (\ref{eqn:contraHyp}) to find a contradiction with (\ref{eqn:contraDecay}). For a contradiction to arise, we must know that the energies converge, and in particular that
\begin{equation}\label{eqn:limEn}
\lim_{n\to \infty} G_n(0,\tau) =  \int_{B(0,\tau)} \CC{e(u)}\dd{x}.
\end{equation} Showing that the left-hand side of \eqref{eqn:limEn} exceeds the limit's  elastic \EEE energy follows from lower semicontinuity associated with weak convergence, but showing  the reverse inequality \EEE  will require a careful choice of competitors in the quasiminimality condition for $(u_n,K_n)$, see (\ref{eq_david_quasi}). The main technical tool to accomplish this is a more regular approximation of $u_n$ given by $\tilde{u}_n$ such that $\tilde{u}_n - a_n \to u$ in $L^2_{\mathrm{loc}}$ for a sequence of rigid motions $a_n$. We remark that  in general \EEE it is not possible to show that {$u_n - a_n$} converges in $L^2$ to the limit, as small inclusions in the domain may be cut out by the crack, and subsequently travel off to infinity.

We use the approximation technique introduced by {\sc Chambolle et al.} in \cite{CCI} for $GSBD$ functions (see also \cite{CFI}). We state the result for pairs $(u,K)$, and note that the analogous statement for $GSBD$ functions may be found by replacing $K_n$ by $J_{u_n}$. We emphasize that in the $GSBD$ setting,  except for \EEE the statement being on  a \EEE  ball (and not a cube), it is as in \cite[Corollary 4.1]{CCI}. We defer the discussion of the proof of this result to Section \ref{sec:Approx}.

\begin{lemma}\label{lem_compactness}
    Let $(u_n,K_n)$ be a sequence of pairs on $B(0,1)$ such that
    \begin{equation*}
        \sup_{ n \in \N} \int_{B(0,1) } \abs{e(u_n)}^2 \dd{x} < \infty\quad \text{ and } \quad  \HH^{N-1}(K_n) \to 0.
    \end{equation*}
    Let $t \in (0,1]$. \EEE   After extracting a subsequence, there exist a sequence of pairs $(\tilde{u}_n,\tilde K_n)$ {in $B(0,1)$}, a sequence of rigid motions $a_n$, and a function $u \in W^{1,2}(B(0, t \EEE );\R^N)$ such that 
    \begin{enumerate}
    \item $\lbrace    \tilde{u}_n \neq u_n\rbrace \subset \subset B(0, t \EEE  )$ and $K_n\triangle \tilde K_n\subset\subset B(0,  t \EEE  )$;
        \item $\tilde{u}_n - a_n \to u$ in $L^2_{\mathrm{loc}}(B(0, t \EEE  );\R^N)$ and $u_n - a_n \to u$ locally in measure in $B(0, t \EEE  )$;
        \item for all  $r < s < t$, \EEE we have
            \begin{equation*}
                \HH^{N-1}\big(\tilde K_n \cap  A_{s,t} \EEE \big) \leq (1 + o(1)) \HH^{N-1}\big(K_n \cap  A_{s,t} \EEE\big)
            \end{equation*}
            and
            \begin{equation*}
                \int_{ A_{s,t}} \CC{e(\tilde{u}_n)} \dd{x} \leq (1 + o(1)) \int_{ A_{r,t}} \CC{e(u_n)} \dd{x},
            \end{equation*}
            where  $A_{s,t} := B(0,t) \setminus B(0,s)$; \EEE
        \item  for all $s \in (0,t]$ we have \EEE
            \begin{equation*}
                \int_{B(0,s)} \CC{e(u)} \dd{x} \leq \liminf_{n \to \infty} \int_{B(0,s)} \CC{e(u_n)} \dd{x}.
            \end{equation*}
    \end{enumerate}
\end{lemma}

We will need to apply this result on nested balls, and it will be important to ensure that the limit function $u$ and the sequence of affine maps can be chosen to be same.

\begin{remark}[Compactness on nested balls]\label{rmk:nestBalls}
{\normalfont
Applying Lemma \ref{lem_compactness} in $B(0,1)$ {(with $t\equiv 1$)} and subsequently in \BBB $B(0, t )$ \EEE for some $ t \EEE  \in (0,1)$, we find (corresponding to a subsequence of a subsequence) functions $\tilde{u}_n$, $a_n$, and $u$ along with $\tilde{u}_n^t$, $a_n^t$, and $u^t$, respectively. We show that we may take $a_n^t = a_n$ and $u^t = u$.

We know that $u_n - a_n$ converges to $u$ in measure and $u_n - a_n^t$ converges to $u^t$ in measure in $B(0,t/2)$, so $a_n^t - a_n$ converges to $u - u^t$ in measure in $B(0,t/2)$.
As $a_n^t - a_n$ is a sequence of affine maps, convergence in measure actually implies convergence in $L^2(B(0,1);\mathbb{R}^N)$, see for instance   \cite[Lemma 3.7]{FLS}. The limit of $a_n^t- a_n$ is necessarily a rigid motion, denoted by $a$, and we have $u = u^t + a$ in $B(0,t)$.
It follows that $\tilde{u}_n^t - a_n \to u$ in $L^2_{\mathrm{loc}}(B(0,t);\R^N)$.
We conclude that, even if we apply the lemma in different nested balls, we can always take the same rigid motions $a_n$ and the same limit function $u$.}
\end{remark}

We can now investigate limits of quasiminimizing sequences with vanishing jump sets. 
The proof is close to \cite[Proposition 3.4]{CFI} and \cite[Theorem 4]{CCI}, but we allow for different weights $M$ and $M^{-1}$ on the surface energy. In addition, we account for the topological {constraint {within the} minimality condition} using a trick observed in the scalar setting by {\sc Sideau} \cite{Si}, see also \cite[Chapter 72]{David}. When we later apply the proposition, the constant $c_n$ will arise from a renormalization of the Griffith energy. As before, we state the result for pairs $(u_n,K_n)$, but the same statement holds for $GSBD$ quasiminimizers when one replaces $K_n$ by $J_{u_n}$.

\begin{proposition}\label{prop_compactness}
    Let $(u_n,K_n)$ be a sequence of pairs {in $B(0,1)$} with vanishing crack, i.e., $\HH^{N-1}(K_n) \to 0$.
    We assume that there exist some $M \geq 1$, a sequence $c_n \geq 0$, and a sequence $\varepsilon_n \to 0$ such that
    \begin{equation}\label{eqn:bddAssump}
        \sup_{ n \in \N} \left(\int_{B(0,1)} \abs{e(u_n)}^2 \dd{x} + c_n \HH^{N-1}(K_n)\right) < + \infty,
    \end{equation}
    and  that \EEE for all $n\in \N$, all $t \in (0,1)$, and all topological competitors $(v,L)$ of $(u_n,K_n)$ in $B(0,t)$, we have
    \begin{align}
        \int_{B(0,t)} & \CC{e(u_n)} \dd{x} + c_n M^{-1} \HH^{N-1}\big(K_n \cap B(0,t)\big) \nonumber
        \\ & \leq \int_{B(0,t)} \CC{e(v)} \dd{x} + c_n M \HH^{N-1}\big(L \cap B(0,t)\big) + \varepsilon_n. \label{eqn:qminProp}
    \end{align}
    Then, up to extracting a subsequence, there exists a function $u \in W^{1,2}(B(0,1);\R^N)$ such that for all $t \in (0,1)$,
    \begin{equation}\label{eqn:compconc1}
    \begin{aligned}
                &\lim_{n \to \infty} \int_{B(0,t)} \CC{e(u_n)} \dd{x} = \int_{B(0,t)} \CC{e(u)} \dd{x},\\
                &\lim_{n \to \infty} c_n \HH^{N-1}\big(K_n \cap B(0,t)\big) = 0,
    \end{aligned}
    \end{equation}
    and such that
    \begin{equation} \label{eqn:compconc2}
        \int_{B(0,1)} \CC{e(u)} \dd{x} \leq \int_{B(0,1)} \CC{e(v)} \dd{x}
    \end{equation}
    for all $v \in u + W^{1,2}_0(B(0,1);\R^N)$.
\end{proposition}
\begin{proof}
\textit{Step 1 (Approximating functions and subsequences).}
    Since the functions
    \begin{align*}
        \alpha_n &: t \mapsto \int_{B(0,t)} \CC{e(u_n)} \dd{x}\\
        \beta_n &: t \mapsto c_n \HH^{N-1}\big(K_n \cap B(0,t)\big)
    \end{align*}
    are non-decreasing and uniformly bounded on $[0,1]$ by \eqref{eqn:bddAssump}, we may find two non-decreasing functions $\alpha, \beta : [0,1] \to \R_+$ and extract a subsequence (not relabeled) such that for all $t \in [0,1]$ it holds that
      \begin{subequations}
    \begin{equation}\label{for the end1}
        \lim_{n \to \infty} \int_{B(0,t)} \CC{e(u_n)} \dd{x} = \alpha(t),
        \end{equation}
            \begin{equation}\label{for the end2}
        \lim_{n \to \infty} c_n \HH^{N-1}\big(K_n\cap B(0,t)\big) = \beta(t).
  \end{equation}
   \end{subequations}
   By Lemma \ref{lem_compactness} and Remark \ref{rmk:nestBalls}, there exist a sequence of rigid motions $(a_n)_n$ and a Sobolev function $u \in W^{1,2}(B(0,1);\R^N)$ such that 
  
    \begin{equation}\label{eqn:alphaIneq}
        \int_{B(0,t)} \CC{e(u)} \dd{x} \leq \alpha(t) \quad \text{  for all $t \in (0,1]$,}
    \end{equation}
    and for all $t \in (0,1]$, we can extract a further subsequence (which depends on $t$ and is not relabeled)  to \EEE find pairs $(\tilde u_n,\tilde K_n)$ {in $B(0,1)$} satisfying \EEE
    \begin{enumerate}
    \item \label{eqn:item1} $\lbrace    \tilde{u}_n \ne u_n\rbrace \subset \subset  B(0,t) \EEE $ and $\tilde K_n \triangle K_n \subset\subset  B(0,t) \EEE $;
        \item the functions $\tilde{u}_n - a_n$ converge to $u$ in $L^2_{\mathrm{loc}}(B(0,t);\R^N)$;
        \item for all $r,s$ such that {$r < s < t$}, we have
            \begin{subequations}
                \begin{equation}\label{eq_cv_tildeun1}
                    \limsup_{n \to \infty} \int_{B(0,t) \setminus B(0,s)} \CC{e(\tilde{u}_n)} \dd{x} \leq \alpha(t) - \alpha(r)
                \end{equation}
                and
                \begin{equation}\label{eq_cv_tildeun2}
                    \limsup_{n \to \infty}  c_n \EEE \HH^{N-1}\big(\tilde K_n \cap B(0,t) \setminus B(0,s)\big) \leq \beta(t) - \beta(s).
                \end{equation}
            \end{subequations}
    \end{enumerate}

\textit{Step 2 (Construction of competitors).}
    Now we fix $t \in (0,1)$ as a point of continuity of $\alpha$ and $\beta$ (almost every $t$ works because these functions are non-decreasing) and let $(\tilde{u}_n,\tilde K_n)$ be as above.
    We let $v \in W^{1,2}(B(0,1);\R^N)$ be such that $\set{v \ne u} \subset \subset B(0,t)$.
    In \EEE particular, there exists $s < t$ such that $\set{v \ne u} \subset B(0,s)$. We let $\varphi \in C^\infty_c(B(0,t))$ be a cut-off function such that $\varphi = 1$ on $B(0,s)$, and we introduce the ($s$-dependent) competitor $(v_n,L_n)$ given by
    \begin{align}\label{v:n}
        v_n  := \varphi (v + a_n) + (1 - \varphi) \tilde{u}_n, \quad \quad \quad 
        L_n  :=  \tilde{K}_n \setminus B(0,s).
    \end{align}
    We are going to estimate the energy of the competitors, i.e., 
    \begin{equation*}
        \limsup_{n \to \infty} \int_{B(0,t)} \CC{e(v_n)} \dd{x} + M \HH^{N-1}(L_n \cap B(0,t)),
    \end{equation*}
    and subsequently modify each competitor $(v_n,L_n)$ to ensure that it is a topological competitor, see (\ref{eq_topological}).
    
\textit{Step 3 (Elastic energy of the competitors).}    
    We have
    \begin{equation*}
        e(v_n) = \varphi e(v) + (1 - \varphi) e(\tilde{u}_n) + \nabla \varphi \odot  (u+a_n - \tilde{u}_n), \EEE
    \end{equation*}
     where we use that $v= u$ on $\lbrace \nabla \varphi \neq 0 \rbrace$. \EEE
    The function $\C \xi : \xi$ comes from a scalar product on the space $\R^{N \times N}_{\mathrm{sym}}$ of symmetric matrices and it is temporarily more convenient to work with the underlying norm, so   for $\xi \in \R^{N \times N}_{\mathrm{sym}}$  we set \EEE
    \begin{equation*}
        \A{\xi} := \sqrt{\C \xi : \xi}.
    \end{equation*}
    By the triangle  inequality \EEE  we have
    \begin{equation*}
        \A{e(v_n)} \leq \varphi \A{e(v)} + (1 - \varphi) \A{e(u_n)}+ C \abs{\nabla \varphi} \abs{\tilde{u}_n - a_n - u}
    \end{equation*}
    for some constant $C \geq 1$ that depends on $\C$.
    It is elementary to check that 
    \begin{equation}\label{eq_elementary}
        (a + b)^2 \leq (1 + \varepsilon) a^2 + (1 + \varepsilon^{-1}) b^2 \quad \text{for all $a,b \geq 0$ and for all $\varepsilon > 0$}.
    \end{equation}
    Using \eqref{eq_elementary} and the convexity of $t \mapsto t^2$, we estimate
    \begin{align*}
        \begin{split}
            \int_{B(0,t)} \A{e(v_n)}^2 \dd{x} &\leq \int_{B(0,t)} \Big( (1 + \varepsilon)  (\varphi \A{e(v)} + (1 - \varphi) \A{e(\tilde{u}_n)})^2 +  C (1 + \varepsilon^{-1}) \abs{\nabla \varphi}^2 \abs{\tilde{u}_n - a_n - u}^2  \Big)\dd{x}
        \end{split}\\
        \begin{split}
                                           &\leq (1 + \varepsilon) \int_{B(0,t)} \varphi \A{e(v)}^2 + (1 - \varphi) \A{e(\tilde{u}_n)}^2 \dd{x}\\
                                           &\qquad + C (1 + \varepsilon^{-1}) \int_{B(0,t)} \abs{\nabla \varphi}^2 \abs{\tilde{u}_n - a_n - u}^2 \dd{x},
        \end{split}
    \end{align*}
    for a bigger constant $C$.
    Then, by construction of $\varphi$, this simplifies to
    \begin{align*}
        \int_{B(0,t)} \A{e(v_n)}^2 \dd{x} \leq & (1 + \varepsilon) \int_{B(0,t)} \A{e(v)}^2 \dd{x} + (1 + \varepsilon) \int_{B(0,t) \setminus B(0,s)} \A{e(\tilde{u}_n)}^2 \dd{x} \\
        & + C (1 + \varepsilon^{-1}) \int_{B(0,t)} \abs{\nabla \varphi}^2 \abs{\tilde{u}_n - a_n - u}^2 \dd{x}.
    \end{align*}
    Returning to the notation $\CC{\xi}$ and using $\tilde{u}_n - a_n \to u$   in $L^2_{\mathrm{loc}}(B(0,t);\R^N)$, $\varphi\in C^\infty_c(B(0,t)) $, \EEE along with \eqref{eq_cv_tildeun1}, this yields, for all {$r,s$ such that} $0 < r < s  < t \EEE$,
    \begin{equation*}
        \limsup_{n \to \infty} \int_{B(0,t)} \CC{e(v_n)} \dd{x} \leq (1 + \varepsilon) \int_{B(0,t)} \CC{e(v)} \dd{x} + (1 + \varepsilon) (\alpha(t) - \alpha(r)).
    \end{equation*}
    In the previous inequality, $\varepsilon$ can be sent to $0$ and $r\uparrow s$ to find
    \begin{equation}\label{eqn:ElastEst}
        \limsup_{n \to \infty} \int_{B(0,t)} \CC{e(v_n)} \dd{x} \leq  \int_{B(0,t)} \CC{e(v)} \dd{x} + (\alpha(t) - \alpha(s-)),
    \end{equation}
     where $\alpha(s-) := \lim_{r \uparrow s} \alpha (r)$. \EEE
    
    \textit{Step 4 (Surface energy of the competitors).} 
 By definition of $L_n$ we have
    \begin{equation*}
        L_n \cap B(0,t) = \tilde K_n \cap B(0,t) \setminus B(0,s).
    \end{equation*}
 Thus, \eqref{eq_cv_tildeun2}  gives \EEE 
    \begin{equation}\label{eqn:Lest}
        \limsup_{n \to \infty} c_n \HH^{N-1}(L_n \cap B(0,t)) \leq \beta(t) - \beta(s).
    \end{equation}

    \textit{Step 5  (Topological \EEE constraint).} 
    The pair $(v_n,L_n)$ may not be a topological competitor of $(u_n,K_n)$ in $B(0,t)$ because we have wiped out $K_n$ in the interior, possibly creating holes. To rectify this, we fix a parameter $t'\in (t,1)$, which will later be sent down to $t$, and modify $L_n$ to construct a topological competitor in $B(0,t')$. \BBB We emphasize that there is no need for this step in the case of $GSBD$ quasiminimizers because topological competitors are not considered in  that case. \EEE
   
   Let $(V_i)_{i \in \N_0}$ denote the open and connected components of $B(0,t') \setminus K_n$, ordered in such a way that $\abs{V_i} \geq \abs{V_{i+1}}$ (we omit the dependence on $n$).
    We show that the  volume of $V_0 $ is big \EEE in the sense that
    \begin{equation}\label{eq_isoperimetric}
        \abs{B(0,t') \setminus V_0} \leq C \HH^{N-1}(K_n)^{N/(N-1)},
    \end{equation}
    where $C>0$ is a  dimensional \EEE constant.  Indeed, \EEE for each $i$, we have $\partial V_i \cap B(0,t') \subset K_n$ and $\HH^{N-1}(K_n) < +\infty$, so by \cite[Proposition 3.62]{AFP} $V_i$ has finite perimeter in $B(0,t')$ with $\partial^* V_i \cap B(0,t') \subset K_n$.
    As $(V_i)_{i \in \N_0}$ forms a Caccioppoli partition of $B(0,t')$, by \cite[Theorem 4.17]{AFP} it follows that
    \begin{equation}\label{eq_ViKn}
        \sum_{ i \ge 0 \EEE} \HH^{N-1} \big( \partial^* V_i \cap  B(0,t') \big)  \le \EEE 2 \HH^{N-1} \big( K_n  \cap  B(0,t') \big). 
    \end{equation} 
    From the ordering properties of the sequence, we see that for all $i \geq 1$,
    \begin{equation*}
        \abs{V_i} \leq \abs{V_0} \leq \abs{B(0,t') \setminus V_i},
    \end{equation*}
    and thus we can estimate by the  relative \EEE isoperimetric inequality
    \begin{equation*}
        \abs{V_i} \leq C \HH^{N-1}\big(\partial^* V_i \cap B(0,t') \big)^{N/(N-1)}.
    \end{equation*}
    We sum the above inequality over $i \geq 1$, use the superadditivity of $t \mapsto t^{N/(N-1)}$,
     and \EEE  apply \eqref{eq_ViKn} to obtain the desired estimate \eqref{eq_isoperimetric}.

    We define the relatively closed set $F_n := B(0,t') \setminus V_0$ and note by the coarea formula that 
    \begin{equation*}
        \abs{F_n} = \int_0^{ t'} \HH^{N-1}\big(F_n \cap \partial B(0,\rho)\big) \dd{\rho}.
    \end{equation*}
    Using this equality and \eqref{eq_isoperimetric}, we find a radius $\rho_n \in (t,t')$ such that
    \begin{equation}\label{eqn:Fest}
        \HH^{N-1}\big(F_n \cap \partial B(0,\rho_n)\big) \leq (t'-t)^{-1} \abs{F_n} \leq {C(t')} \HH^{N-1}(K_n)^{N/(N-1)}.
    \end{equation}
    We then replace $L_n$ by
    \begin{equation*}
        G_n := L_n \cup \big(F_n \cap \partial B(0,\rho_n)\big).
    \end{equation*}
    Since $c_n \HH^{N-1}(K_n)$ is uniformly bounded and $\HH^{N-1}(K_n) \to 0$, we use \eqref{eqn:Fest} to deduce that
    \begin{equation}\label{eq_Fn}
        \lim_{n \to \infty} c_n \HH^{N-1}\big(F_n \cap \partial B(0,\rho_n)\big) = 0,
    \end{equation}
    hence the contribution of $F_n$ to the crack will disappear in the limit.
    
    We now justify that $(v_n,G_n)$ is a topological competitor of $(u_n,K_n)$ in the ball $B(0,t')$ and, in particular, satisfies (\ref{eq_topological}).
    Arguing by contraposition, we consider two points $x$ and $y$ in $ B(0,1) \setminus  ( B(0,t') \cup K_n ) \EEE $ that are connected by a continuous path which is disjoint from $G_n$.
    If the path does not intersect $\partial B(0,\rho_n)$, then it stays in the complement of $\overline{B(0,\rho_n)}$ where $G_n$ coincides with $K_n$. Therefore, this continuous path does not meet $K_n$, and $x$ and $y$ are also connected in $B(0,1)\setminus K_n$. If the path meets $\partial B(0,\rho_n)$, it can only be at a point of $V_0$. Considering the portion of the path starting from $x$ to the first time it meets $\partial B(0,\rho_n)$, we see that $x$ is connected to $V_0$ in the complement of $K_n$. Likewise, we see that $y$ is also connected to $V_0$ in the complement of $K_n$. Since $V_0$ is an open connected set disjoint from $K_n$, $x$ and $y$ are also connected in $B(0,1)\setminus K_n$. This shows that (\ref{eq_topological}) holds, as desired.
    
  \textit{Step 6  (Limit passage \EEE using quasiminimality).}   
 By item (\ref{eqn:item1}) in Step 1 and \eqref{v:n} we get   $e(v_n) = e(u_n)$ a.e.\ in $B(0,t') \setminus B(0,t)$. This along with     \EEE \eqref{for the end1} and (\ref{eqn:ElastEst})  gives 
    \begin{equation*}
        \limsup_{n \to \infty} \int_{B(0,t')} \CC{e(v_n)} \dd{x} \leq \int_{B(0,t)} \CC{e(v)} \dd{x} + (\alpha(t') - \alpha(s-)){.}
    \end{equation*}
      Similarly, by  \eqref{for the end2}, \EEE  (\ref{eqn:Lest}),   and \eqref{eq_Fn}, \EEE     
    we have
    \begin{equation*}
        \limsup_{n \to \infty} c_n \HH^{N-1}\big(G_n \cap B(0,t')\big) \leq \beta(t') - \beta(s).
    \end{equation*}
    Finally, we apply the quasiminimality property of $(u_n,K_n)$ in $B(0,t')$, see (\ref{eqn:qminProp}), to find
    \begin{align*}
        \int_{B(0,t')} & \CC{e(u_n)} \dd{x} + c_n M^{-1} \HH^{N-1}\big(K_n \cap B(0,t')\big) \\
        & \leq \int_{B(0,t')} \CC{e(v_n)} \dd{x} + c_n M \HH^{N-1}\big(G_n \cap B(0,t')\big) + \varepsilon_n.
    \end{align*}
    Taking the $\limsup$ of both sides  and using \eqref{for the end1}--\eqref{for the end2}, \EEE we find
    \begin{equation*} 
        \alpha(t') + M^{-1} \beta(t') \leq \int_{B(0,t)} \CC{e(v)} \dd{x} + (\alpha(t') - \alpha(s-)) +  M \EEE (\beta(t') - \beta(s)).
    \end{equation*}
    As $t$ is a point of continuity for both $\alpha$ and $\beta$, we can  let \EEE  $t'\downarrow t$ and $s\uparrow t$ to recover
    \begin{equation}\label{eq_energy_limit2}
        \alpha(t) + M^{-1} \beta(t) \leq \int_{B(0,t)} \CC{e(v)} \dd{x}.
    \end{equation}
    {This holds a priori for almost all $t \in (0,1)$ and all $v \in W^{1,2}(B(0,1);\R^N)$ such that $\set{u \ne v} \subset \subset B(0,1)$.}
    Since $\alpha$ and $\beta$ are {non-decreasing} {and the right-hand side of (\ref{eq_energy_limit2}) is continuous in $t$}, we deduce that (\ref{eq_energy_limit2}) actually holds for all $t \in (0,1)$.
    Taking $v = u$ and recalling (\ref{eqn:alphaIneq}), we find that  
    $$\alpha(t) = \int_{B(0,t)} \CC{e(u)} \dd{x} \quad \quad \text{and} \quad \quad \beta(t) = 0 $$ for all $t\in (0,1)$.  This along with \eqref{for the end1}--\eqref{for the end2} shows \EEE \eqref{eqn:compconc1}. Using these two equalities in \eqref{eq_energy_limit2}, we have for all $v \in W^{1,2}(B(0,1);\R^N)$ with $\set{v \ne u} \subset \subset B(0,t)$ that
    \begin{equation*}
        \int_{B(0,t)} \CC{e(u)} \dd{x} \leq \int_{B(0,t)} \CC{e(v)} \dd{x}.
    \end{equation*}
  Then, \EEE   it follows easily that $u$ minimizes the elastic energy $v \mapsto \int_{B(0,1)} \CC{e(v)} \dd{x}$ subject to its own Dirichlet condition on $\partial B(0,1)$, giving \eqref{eqn:compconc2} and concluding the proof.
\end{proof}

 Before turning to the proof of the decay estimate, we observe \EEE  that quasiminimizers have a natural scaling coming from the energy.
 
\begin{remark}[Rescaling]\label{rmk_scaling}
{\normalfont 
    If a pair $(u,K)$ is a topological $M$-quasiminimizer in $B(x_0,r_0)$ with gauge $h$, then the pair $(\hat{u},\hat{K})$ defined by
    \begin{equation*}
        \hat{u}(x) := r_0^{-1/2} u(x_0 + x r_0) \quad \text{and} \quad  \hat{K} := r_0^{-1} (K - x_0)
    \end{equation*}
    is a topological $M$-quasiminimizer in $B(0,1)$, with gauge $\hat{h}(t) := h(r_0 t)$. Moreover, one has for all $t \in (0,1]$,
    \begin{align*}
        \int_{B(0,t) \setminus \hat{K}} &\CC{e(\hat{u})} \dd{x} + \HH^{N-1}\big(\hat{K} \cap B(0,t)\big) \\
        & = \frac{1}{r_0^{N-1}} \left(\int_{B(x_0,t r_0) \setminus K} \CC{e(u)} \dd{x} + \HH^{N-1}\big(K \cap B(x_0,t r_0)\big)\right).
    \end{align*}
    The same scaling is at play for $GSBD$  $M$-quasiminimizers, except it only requires defining $\hat{u}$ as above, as this implicitly includes the rescaling for the jump set $J_u.$  }
\end{remark}
We finally pass to the proof of the decay estimate for the energy.

\begin{proof}[Proof of Proposition \ref{prop_decay}]
    By scaling (Remark \ref{rmk_scaling}), it suffices to prove the result in $B(0,1)$.
    We proceed by contradiction and assume that for some fixed $M \geq 1$, $\tau \in (0,1)$, and $C>0$, there exist a sequence $\varepsilon_n \to 0$ and a sequence of pairs $(u_n,K_n)$ such that each $(u_n,K_n)$ is $M$-quasiminimal with gauge $h_n$ in $B(0,1)$ and
    \begin{equation*}
        \HH^{N-1}(K_n \cap B(0,1)) \leq \varepsilon_n, \quad h_n(1) \leq \varepsilon_n G_n(0,1)
    \end{equation*}
    but
    \begin{equation}\label{eq_Gk}
        G_n(0,\tau) > C \tau^N G_n(0,1),
    \end{equation}
    where
    \begin{align*}
        G_n(0,t) &:= \int_{B(0,t)} \CC{e(u_n)} \dd{x} + \HH^{N-1}\big(K_n \cap B(0,t)\big)\quad \text{ for all } t \in (0,1).
    \end{align*}
    We will derive a contradiction by choosing $C>0$ large enough (depending only on $N$, $\C$ but not  on \EEE $\tau$ or $M$).
    We set $g_n := G_n(0,1)$ and observe using \eqref{eq_Gk} that {$g_n \in (0,\infty)$}.
    We then define $w_n := g_n^{-1/2} u_n$. One can check that
    \begin{equation}\label{eq_vk_energy}
        \int_{B(0,1)} \CC{e(w_n)} \dd{x} + g_n^{-1} \HH^{N-1}\big(K_n \cap B(0,1)\big) = 1,
    \end{equation}
    \begin{equation}\label{eq_tau}
        \int_{B(0,\tau)} \CC{e(w_n)} \dd{x} + g_n^{-1} \HH^{N-1}\big(K_n \cap B(0,\tau)\big) > C \tau^N
    \end{equation}
    and {that} for all $t \in (0,1)$, for all pairs $(v,L)$ with $\lbrace   v \ne w_n\rbrace \subset \subset B(0,t)$ and $L\triangle K_n \subset\subset B(0,t)$, we have
    \begin{multline*}
        \int_{B(0,t)} \CC{e(w_n)} \dd{x} + M^{-1}  g_n^{-1} \EEE  \HH^{N-1}\big(K_n \cap B(0,t)\big) \\\leq \int_{B(0,t)} \CC{e(v)} \dd{x} + M g_n^{-1} \EEE \HH^{N-1}\big(L \cap B(0,t)\big) + \varepsilon_n.
    \end{multline*}
   By an application of Proposition \ref{prop_compactness}  (for $c_n = g_n^{-1}$), \EEE we find a subsequence (not relabeled) and a function \\ ${w \in W^{1,2}(B(0,1);\R^N)}$ such that for all $t \in  (0,1)$, \EEE
    \begin{align*}
                &\lim_{n \to \infty} \int_{B(0,t)} \CC{e(w_n)} \dd{x} = \int_{B(0,t)} \CC{e(w)} \dd{x},\\
                &\lim_{n \to \infty} g_n^{-1} \HH^{N-1}(K_n \cap B(0,t)) = 0,
    \end{align*}
    and $w$ is a minimizer of the elastic energy $v \mapsto \EEE \int_{B(0,1)} \CC{e(v)}\dd{x}$ subject to its own boundary conditions.
     Thanks to Korn's inequality, the elasticity tensor $\C$ is coercive in the sense that $\norm{\varphi}^2_{H^1} \leq C \int \CC{e(\varphi)}$ for all test functions $\varphi$. This allows the application of the standard regularity theory of elliptic systems \cite[Section 4 and 5]{giaquinta}, and in particular there exist a constant $C_0 \geq 1$ (depending on $N$, $\C$) such that \EEE
    \begin{equation}\label{eqn:ellipDecay}
        \int_{B(0,\tau)} \CC{e( w \EEE )} \dd{x} \leq C_0 \tau^{N} \int_{B(0,1)} \CC{e(  w \EEE )} \dd{x}.
    \end{equation}
    But passing to the limit in \eqref{eq_vk_energy} and \eqref{eq_tau} yields
    \begin{align*}
        \int_{B(0,1)} \CC{e(  w \EEE )} \dd{x}   \le \EEE 1, \quad \quad \quad           \int_{B(0,\tau)} \CC{e( w \EEE )} \dd{x}  \geq C \tau^N.
    \end{align*}
    Plugging these into (\ref{eqn:ellipDecay}), we obtain a contradiction by taking $C = 2 C_0$.
\end{proof}

\section{Approximation of pairs with a small crack set}\label{sec:Approx}
\newcommand{\ds}{d}

The goal of this section is to establish the compactness Lemma \ref{lem_compactness} for pairs $(u,K)$ and $GSBD$ functions. This was shown by {\sc Chambolle et al.} \cite{CCI} for $GSBD$ functions on the unit cube. Here we discuss how the proof can be adapted to the unit ball and, for topological quasiminimizers, how one ensures that the constructed approximations $(\tilde u_n, \tilde K_n)$ are \textit{admissible} in the sense that $\tilde u_n$ is Sobolev away from the crack and that the constructed crack is closed. 

As in \cite{CCI}, an essential step in the proof is a mollification procedure to construct an approximation of a function with small crack. We adapt this approach to our setting and emphasize the necessary changes to account for (1) working on the unit ball and (2) ensuring that the pair $ (\tilde u_n,\tilde K_n)$ \EEE is admissible. These are the only changes in the proof, and they do not deal with the topological nature of the competitors in any other way.  Given that our result is an adaptation of  the original technical proof for $GSBD$ in \cite[Theorem 3]{CCI}, we \EEE assume that the reader interested in these details is already familiar  with the proof in \cite{CCI}, \EEE and we discuss how the proof can be adapted to our setting.

\newcommand{\diam}[1]{\ell_{#1}}
\begin{proposition}\label{prop_approximation}
    There exist   constants $\eta > 0$, $c > 1$, $s > 0$  (depending on $N$ and $\C$) \EEE such that the following holds.
    Let $(u,K)$ be a pair in $B(0,1)$ with 
    \begin{equation*}
        \int_{B(0,1)  } \abs{e(u)}^2 \dd{x} + \HH^{N-1}(K) < +\infty.
    \end{equation*}
    Assume that $\delta := \HH^{N-1}(K)^{1/N}  < \EEE \eta$.   Then, \EEE there exist a radius $R \in (1 - \sqrt{\delta},1)$, a pair $(\tilde{u},\tilde{K})$ in $B(0,1)$, and a Borel set $\tilde{\omega} \subset B(0,R)$ such that
    \begin{enumerate}
        \item \label{lab:K1} $\tilde{u}$ is $C^\infty$ in $B(0,1-\sqrt{\delta})$ and $(\tilde{u},\tilde{K}) = (u,K)$ in $B(0,1) \setminus B(0,R)$;
        \item \label{lab:K2} we have
            \begin{equation}\label{the crack}
                \HH^{N-1}(\tilde{K} \setminus K) \leq c \sqrt{\delta} \HH^{N-1}\big(K \setminus B(0,1-\sqrt{\delta})\big)
            \end{equation}
            and for all open set $\Omega \subset B(0,1)$,
            \begin{equation}\label{the elastic}
                \int_{\Omega} \CC{e(\tilde{u})} \dd{x} \leq (1+c \delta^s)\int_{\Omega_{\delta}} \CC{e(u)} \dd{x} + c  \sqrt{\delta}  \int_{B(0,1)\setminus B(0,1-\sqrt{\delta})}  \CC{e(u)} \dd{x},
            \end{equation}
            where $\Omega_{\delta} = \set{x \in B(0,1) \colon \mathrm{dist}(x,\Omega) < c\delta}$;  
        \item \label{lab:K3} we have
            \begin{equation*}
                \abs{\tilde{\omega}} \leq c \delta \HH^{N-1}(K)  
            \end{equation*}
            and
            \begin{equation*}
                \int_{B(0,1) \setminus \tilde{\omega}} \abs{\tilde{u} - u}^2 \dd{x} \leq c \delta^2 \int_{B(0,1)} \abs{e(u)}^2 \dd{x}.
            \end{equation*}
    \end{enumerate}
\end{proposition}

\begin{proof}[Outline of proof]
We outline the proof and take the liberty to omit necessary details spelled out in \cite[Theorem 3]{CCI}, and instead emphasize why the setting of pairs only requires a couple  of \EEE adaptations.

\textit{Step 1 (Whitney covering).}
To  define \EEE   the approximation $(\tilde u, \tilde K)$, the idea is to find a good radius $R \in (1-\sqrt{\delta},1)$ via an averaging argument with energetically favorable properties.  In particular, $R$ is chosen such that \EEE
\begin{align}\label{an R choice}
\mathcal{H}^{N-1}\big(K \cap \partial B(0,R)\big) = 0.
\end{align}  
Then, we will make no changes to $(u,K)$ outside of the ball $B(0,R)$. To modify $(u,K)$ inside $B(0,R)$, we construct a Whitney covering of this region made up of cubes $\{q_i\}_{i\in \N}$ with each side of $q_i$ is of length $r_i$. Further, we assume $q_i' \subset\subset B(0,R)$, where $q_i'$ is a cube with the same center as $q_i$ and side length given by $3r_i/2$. Importantly, \BBB all cubes intersecting $B(0,1-\sqrt{\delta})$ have side length $\delta$ and \EEE the cubes refine as they approach $\partial B(0,R)$, meaning ${\rm diam}(q_i)\to 0$ as the center of $q_i$ approaches $\partial B(0,R)$. This is used to ensure that the trace of the modified function coincides with  the trace of \EEE  $u$. In \cite{CCI}, this step is done using an explicit covering made of dyadic cubes.  Working on the ball, as we do, requires  only minor changes, namely using a Whitney covering for the ball. \EEE

\textit{Step 2 (Good and bad cubes).}
We partition the cubes $q_i$ into \textit{good} cubes and \textit{bad} cubes, depending on  whether \EEE the cube contains a relatively large proportion of the crack. Precisely,  fixing $\eta$ as $\eta = \frac{1}{2} 8^{-N} c_*^{-1}$ with $c_*$ from \cite[(6)]{CCI}, \EEE $q_i$ is a \textit{good cube} if 
\begin{equation}\label{eqn:badcube}
\mathcal{H}^{N-1}(K\cap q_i')\leq \eta r_i^{N-1},
\end{equation}
and a \textit{bad cube} otherwise. We denote the indices of good cubes by $\mathbb{G}$ and, similarly, $\mathbb{B}$ for the bad cubes.

\textit{Step 3 (Modified function).}
We consider a partition of unity $\{\phi_i\}_{i\in \mathbb{G}}$ on the good cubes, subordinate to $\bigcup_{i\in \mathbb{G}} q_i'  \setminus \bigcup_{i\in \mathbb{B}} q_i\EEE$. In each good cube $q_i$, we consider a smoothed out function $\rho_{\delta_i}*(u\chi_{q_i'\setminus \omega_i} + a_i \chi_{\omega_i})$, where $\omega_i \subset q_i'$ and $a_i$ is a rigid motion associated to the cube through the Korn inequality for functions with a small jump set  (see \cite[(13)--(14)]{CCI}), \EEE and $\rho_{\delta_i}$ is a mollifier supported in $B(0,\delta_i)$ with $\delta_i \approx r_i/10$. The modified function $\tilde u$ is defined as
\begin{equation}\label{eqn:smoothDef}
\tilde u(x)  := \begin{cases}
\sum_{i\in \mathbb{G}} \phi_i(x) (\rho_{\delta_i}*(u\chi_{q_i'\setminus \omega_i} + a_i \chi_{\omega_i}))(x) & \text{ if }x \in \bigcup_{i\in \mathbb{G}}q_i \setminus \bigcup_{i\in \mathbb{B}}q_i, \\
u(x) & \text{ if }x \in (\bigcup_{i\in \mathbb{B}}q_i) \cup (B(0,1)\setminus B(0,R)).
\end{cases}
\end{equation}
As we have blindly stitched together the mollified function in the good cubes with the original function $u$ on the bad cubes, we cannot prevent a jump in the function across these boundaries. So, the new jump set must be defined as 
\begin{equation}\label{eqn:KtildeDef}
\tilde K = K \cup  \overline{\bigcup\nolimits_{i\in \mathbb{B}}\partial q_i},
\end{equation} \EEE
where we must take the closure to ensure  that \EEE $\tilde K$ is closed.  The assumption $\delta < \eta$ ensures that all cubes intersecting $B(0,1-\sqrt{\delta})$ lie in $\mathbb{G}$, and thus  $\tilde{u}$ is $C^\infty$ in $B(0,1-\sqrt{\delta})$. \BBB Indeed, as all cubes $q_i$ intersecting $B(0,1-\sqrt{\delta})$ have side length $\delta$, we get $\mathcal{H}^{N-1}(K\cap q_i')\leq \delta^N < \eta \delta^{N-1}  =  \eta r_i^{N-1}$, and thus \eqref{eqn:badcube} holds.  \EEE

\textit{Step 4 (Size of the closed crack).}
Estimating the size of the new crack proceeds  as \EEE in the original proof of \cite[Theorem 3]{CCI}, except one must now additionally make sure that there is no charge arising on the boundary of $\partial B(0,R)$ coming from  taking the closure of ${\bigcup_{i\in \mathbb{B}}\partial q_i}$ {in \eqref{eqn:KtildeDef}}. \EEE We show that 
\begin{equation}\label{eqn:crackContainment}
\overline{\bigcup\nolimits_{i\in \mathbb{B}}\partial q_i} \cap \partial B(0,R) \subset K \cap \partial B(0,R), 
\end{equation}
which then  by \eqref{an R choice} \EEE gives
$$\mathcal{H}^{N-1}\Big(\overline{\bigcup\nolimits_{i\in \mathbb{B}}\partial q_i} \cap \partial B(0,R)\Big) = 0. $$
To see (\ref{eqn:crackContainment}), if $x\in \overline{\bigcup_{i\in \mathbb{B}}\partial q_i} \cap \partial B(0,R),$ then there is a sequence of bad cubes $(q_{i_k})_{k\in \N}$ with $i_k\in \mathbb{B}$ and ${\rm dist}(x,q_{i_k})\to 0$ as $k\to \infty$. Since $x\in \partial B(0,R),$ we also have ${\rm diam}(q_{i_k})\to 0$. Finally, as (\ref{eqn:badcube}) fails for each bad cube, $K \cap q_{i_k}\neq \emptyset$, and we may combine the above facts to conclude that ${\rm dist}(x,K)\to 0$. This implies $x \in K \cap \partial B(0,R)$ because $K$ is closed, concluding (\ref{eqn:crackContainment}).  This along with \cite[(12)]{CCI} shows \eqref{the crack}. \EEE

\textit{Step 5 (Locally Sobolev function).}
We need to ensure that $ \tilde{u} \EEE \in W^{1,2}_{\rm loc}(\Omega \setminus \tilde K; \R^N)$. If $x\not \in \tilde K\cup \partial B(0,R)$, as the Whitney covering is locally finite in the interior of $B(0,R)$ and no change is made outside of $B(0,R)$, it is clear that there is a radius $r>0$ such that $ \tilde{u} \EEE \in W^{1,2}(B(x,r);\R^{N})$. By definition of $\tilde K$ in (\ref{eqn:KtildeDef}), if $x\in \partial B(0,R)\setminus \tilde K$, then $x \in \partial B(0,R)\setminus K$  by \eqref{eqn:crackContainment}. \EEE Consequently, there is a radius   $r >0$ \EEE such that $B(x,r)$ does not intersect $K$ or, thereby, any bad cube. As in the original proof, one can then verify  that \EEE the trace of $\tilde u$ coincides with $u$ on $\partial B(0,R)$  implying $ \tilde{u}  \in W^{1,2}(B(x,r);\R^{N})$. \EEE However, there is a simpler option in the setting of pairs $(u,K)$. If $q_i \cap K = \emptyset$, one can simply replace $\rho_{\delta_i}*(u\chi_{q_i'\setminus \omega_i} + a_i \chi_{\omega_i})$ by $u$ in the definition of $\tilde u$ in \eqref{eqn:smoothDef}, so that $\tilde u = u$ in $B(x,r)$ for $r>0$ sufficiently small.  This discussion also shows $(\tilde{u},\tilde{K}) = (u,K)$ in $B(0,1) \setminus B(0,R)$ concluding the proof of \eqref{lab:K1}. \EEE

\textit{Step 6 (Remaining arguments).} The proof of  \eqref{the elastic} and item \eqref{lab:K3} remains the same as in \cite{CCI}, see in particular  \cite[(26)--(27)]{CCI}. \EEE
\end{proof}

Finally, we turn to the proof of the compactness result used to obtain the decay estimate.

\begin{proof}[Proof of Lemma \ref{lem_compactness}]
 By rescaling it suffices to  perform the proof   for $t=1$. \EEE 
    We let $\delta_n := \HH^{N-1}(K_n)^{1/N}$. For all sufficiently large $n$, we may apply Proposition~\ref{prop_approximation}.
    We obtain a sequence of pairs $(\tilde{u}_n,\tilde{K}_n)$ and of sets $\tilde{\omega}_n$ satisfying properties (\ref{lab:K1}), (\ref{lab:K2}), and (\ref{lab:K3}) of Proposition \ref{prop_approximation}.  By (\ref{lab:K1}) we get  $\lbrace   \tilde{u}_n \neq u_n\rbrace \subset \subset B(0,1)$ and $K_n\triangle \tilde K_n\subset\subset B(0,1)$. \EEE  From (\ref{lab:K2}), we observe that for all  $s \in (0,1)$, \EEE
    \begin{equation*}
        \HH^{N-1}(\tilde{K}_n \cap   A_{s,1} \EEE ) \leq \HH^{N-1}(  K_n \EEE \cap  A_{s,1} \EEE ) + c \sqrt{\delta_n} \HH^{N-1}\big(  K_n \EEE  \setminus B(0,1-\sqrt{\delta_n})\big)
    \end{equation*}
    and for $n$ large enough, $ K_n \EEE  \setminus B(0,1-\sqrt{\delta_n}) \subset  K_n \EEE  \cap A_{s,1}$, whence
    \begin{equation*}
        \HH^{N-1}(\tilde{K}_n \cap A_{s,1}) \leq (1 + o(1)) \HH^{N-1}(  K_n \EEE  \cap A_{s,1}). 
    \end{equation*}
    Similarly, for $0 <  r < s  \EEE < 1$, we observe that for all $n$ sufficiently large
    \begin{equation*}
        (A_{s,1})_{\delta_n}  =  \big\{x \in B(0,1) \colon \, \mathrm{dist}(x,A_{s,1}) < c \delta_n\rbrace \subset  A_{r,1} \EEE
    \end{equation*}
    and $B(0,1) \setminus B(0,1-\sqrt{\delta_n}) \subset  A_{r,1}$, \EEE so that again by (\ref{lab:K2}) we have
        \begin{align*}
            \int_{ A_{s,1}} \CC{e(\tilde{u}_n)} \dd{x}   &\leq (1 + c \delta^s) \int_{ A_{r,1}} \CC{e({u}_n)} \dd{x} + c \sqrt{\delta} \int_{ A_{r,1}} \CC{e({u}_n)} \dd{x}\notag\\
                                                    &\leq (1 + o(1)) \int_{ A_{r,1}} \CC{e({u}_n)} \dd{x}. 
        \end{align*}
    According to (\ref{lab:K3}), the sets $(\tilde{\omega}_n)_n \subset B(0,1)$ are such that $\abs{\tilde{\omega}_n} \to 0$ and
    \begin{equation}\label{eq_H3}
        \text{$(\tilde{u}_n - u_n)  \chi_{B(0,1) \EEE \setminus \tilde{\omega}_n} \to 0$ in $L^2(B(0,1);\R^{N})$.}
    \end{equation}
    Property (\ref{lab:K1}) shows that for any fixed $t \in (0,1)$, the functions $\tilde{u}_n$ belong to $C^\infty(B(0,t);\R^N)$ for all $n$ sufficiently large. Since $e(\tilde{u}_n)$ is uniformly bounded in $L^2$ by {(\ref{lab:K2})}, \EEE we deduce that, up to extracting a subsequence, there exist a sequence of rigid motions $a_n$ and a function $u \in W^{1,2}(B(0,t);\R^N)$ such that
    \begin{equation}\label{eq_H1}
        \text{$\tilde{u}_n - a_n \to u$ in $L^2(B(0,t);\R^N)$}
    \end{equation}
    and
    \begin{equation}\nonumber 
        \text{$e(\tilde{u}_n) \rightharpoonup e(u)$ in $L^2(B(0,t);\R^{N\times N})$}.
    \end{equation}
    Using \eqref{eq_H3}, \eqref{eq_H1}, and $|\tilde \omega_n|\to 0$, we conclude that
    \begin{equation*}
        \text{$(u_n - a_n)  \chi_{B(0,1) \EEE \setminus \tilde{\omega}_n} \to u$ in $L^2(B(0,t);\R^{N})$},
    \end{equation*}
    and in particular, $u_n - a_n$ converges to $u$ in measure in $B(0,t)$.

    As in Remark \ref{rmk:nestBalls}, we note that, if we apply the above argument in $B(0,t)$ and subsequently in $B(0,t')$, where {$t' \in (t,1)$}, we may take the sequence of rigid motions $(a_n)_n$ (now associated to a subsequence of a subsequence) and the limit function $u$ to be the same in each ball.  
    Diagonalizing on a sequence $t\uparrow 1$, we may find   $u \in W^{1,2}(B(0,1);\R^N)$  and \EEE a subsequence (not relabeled) such that
    \begin{align*}
        \tilde{u}_n - a_n \to u\text{ in }L^2_{\mathrm{loc}}(B(0,1);\R^{N}) \quad \text{ and  } \quad 
        e(\tilde{u}_n) \rightharpoonup e(u)\text{ in }L^2_{\mathrm{loc}}(B(0,1);\R^{N\times N}),
    \end{align*}
    and $u_n - a_n$ converges locally in measure to $u$ in $B(0,1)$.     Finally, as $e(\tilde{u}_n) \rightharpoonup e(u)$ in $L^2_{\mathrm{loc}}(B(0,1);\R^{N\times N})$, for all  $r \in (0,1)$ \EEE we get
    \begin{equation*}
        \int_{B(0,r)} \CC{e(u)} \dd{x} \leq \liminf_{n \to \infty} \int_{B(0,r)} \CC{e(\tilde{u}_n)} \dd{x},
    \end{equation*}
    and applying \eqref{the elastic} for $\Omega =  B(0,r) \EEE $,
   we also have for all  $s \in (r,1]$ \EEE that
    \begin{equation*}
        \liminf_{n \to \infty} \int_{B(0,r)} \CC{e(\tilde{u}_n)} \dd{x} \leq \liminf_{n\to \infty} \int_{B(0,s)} \CC{e({u}_n)} \dd{x}.
    \end{equation*}
    We therefore find 
    \begin{equation*}
        \int_{B(0,r)} \CC{e(u)} \dd{x} \leq \liminf_{n \to \infty} \int_{B(0,s)} \CC{e(u_n)} \dd{x},
    \end{equation*}
    and we can let $r   \uparrow \EEE s$ on the left-hand side to prove the last item of the  statement.
\end{proof}

\section*{Acknowledgements} 

This work was supported by the DFG project FR 4083/5-1 and by the Deutsche Forschungsgemeinschaft (DFG, German Research Foundation) under Germany's Excellence Strategies EXC 2044 -390685587, Mathematics M\"unster: Dynamics--Geometry--Structure and EXC 2047/1 - 390685813, and project 211504053 - SFB 1060.
The second author was also funded by the French National Research Agency (ANR) under grant ANR-21-CE40-0013-01 (project GeMfaceT).

\appendix

\section{Pointwise definition of \texorpdfstring{$GSBD$}{GSBD} quasiminimizers}\label{sec:appendix}
  In this section, \EEE we justify that a $GSBD$ quasiminimizer can be defined pointwise everywhere outside of the crack set.
    We note that, due to the error terms in the definition of quasiminimizers (coming from the gauge $h$), the displacement field does not solve an elliptic PDE away from the crack set and standard regularity theory does not show that $u$ is smooth.
    In the scalar case, using the upper bound on the elastic energy,  namely \EEE
    \begin{equation*}
        \int_{B(x_0,r)} \abs{\nabla u}^2 \dd{x} \leq C r^{N-1},
    \end{equation*}
    and {\sc Campanato}'s characterization of H\"older spaces, one can deduce that $u$ is H\"older regular in any small ball $B(x,r)$ for which $\mathcal{H}^{N-1}(B(x,r)\cap J_u) = 0$, which leads to the stronger conclusion $B(x,r)\cap J_u = \emptyset$.
    In the vectorial case, {{similar} reasoning shows} that every point outside of the crack set is a Lebesgue point,  see also \cite[Proof of Theorem 5.7, Part 2]{CC}. \EEE 
    As we are interested in the behavior of the displacement $u$ ``away from the crack", we consider a ball $B(x_0,r)$ with $\mathcal{H}^{N-1}(B(x_0,r)\cap J_u) = 0$. By slicing, this implies  that \EEE $u$ is a Sobolev function. By rescaling (see Remark \ref{rmk_scaling})   we may replace $B(x_0,r)$ by the unit ball $B(0,1)$.
    \begin{proposition}\label{prop_campanato}
        Let $u \in W^{1,2}_{\mathrm{loc}}(B(0,1);\R^N)$ and assume that there exists a constant $C \geq 1$ such that for any ball $B(x_0,r) \subset B(0,1)$, we have
        \begin{equation}\label{eq_campanato}
            \int_{B(x_0,r)} \abs{e(u)}^2 \dd{x} \leq C r^{N-1}.
        \end{equation}
        Then, for every $x_0 \in B(0,1)$, there exists $b_{x_0} \in \R^N$ such that
        \begin{equation*}
            \lim_{r \to 0} \fint_{B(x_0,r)} \abs{u - b_{x_0}}^2 \dd{x} = 0.
        \end{equation*}
        In the language of $BV$, this implies {$J_u = S_u = \emptyset$}.
    \end{proposition}
    \begin{proof}
        We use $C$ as a generic constant greater than $1$ that depends only on $N$ and the constant in \eqref{eq_campanato}.
        We recall that for all balls $B(x_0,r) \subset \R^N$ and for all linear  maps $a(x)  = A (x - x_0) + b$, we have
        \begin{equation}\label{haha}
            C^{-1} \left(\abs{b}^2 + r^2 \abs{A}^2\right) \leq \fint_{B(x_0,r)} \abs{a(x)}^2 \dd{x} \leq C \left(\abs{b}^2 + r^2 \abs{A}^2\right).
        \end{equation}
        by equivalence of norms in a finite dimensional space.
        \EEE
        Now, let us fix $x_0 \in B(0,1)$ and $R > 0$ such that $B(x_0,R) \subset B(0,1)$.
        For $0 < r \leq R$,  the Korn-Poincaré inequality and \eqref{eq_campanato} tell us that
        \begin{equation}\label{eq_Bu}
            \fint_{B(x_0,r)} \abs{u - a_{x_0,r}}^2 \dd{x} \leq C r^2 \fint_{B(x_0,r)} \abs{e(u)}^2 \dd{x} \leq C r,
        \end{equation}
        where the rigid motion $a_{x_0,r}(x)   = A_{x_0,r}(x - x_0) +  b_{x_0,r}$ is defined by
        $${A_{x_0,r} : = \fint_{B(x_0,r)} \frac{\nabla u - \nabla u^T}{2} \dd{x} \quad \text{ and }\quad   b_{x_0,r} : =\EEE\fint_{B(x_0,r)} (u(x)-A(x-x_0))\dd{x} .}$$
         We consider $0 < s\leq r \leq R$. If $s \in [r/2,r]$, we estimate
        \begin{equation*}
            \int_{B(x_0,s)} \abs{a_{x_0,s} - a_{x_0,r}}^2 \dd{x} \leq C \int_{B(x_0,s)} \abs{u - a_{x_0,s}}^2 \dd{x} + C \int_{B(x_0,r)} \abs{u - a_{x_0,r}}^2 \dd{x} \leq C r^{N+1}
        \end{equation*}
        and therefore $\abs{b_{x_0,s} - b_{x_0,r}} \leq C r^{1/2}$ and $\abs{A_{x_0,s} - A_{x_0,r}} \leq C r^{-1/2}$  by \eqref{haha}. \EEE 
        Next, for any $s \in (0,r]$, there exists an integer $k \geq 0$ such that $s \in [2^{-(k+1)}r, 2^{-k}r]$.
        According to  the \EEE previous estimate, we have
        \begin{equation*}
            \abs{b_{x_0,s} - b_{x_0,2^{-k}r}} \leq C (2^{-k}r)^{1/2}    
        \end{equation*}
        and, for all $0 \leq i < k$,
        \begin{equation*}
            \abs{b_{x_0,2^{-(i+1)}r} - b_{x_0,2^{-i}r}} \leq C (2^{-i} r)^{1/2},
        \end{equation*}
        whence
        \begin{equation}\label{eq_b_decay}
            \abs{b_{x_0,s} - b_{x_0,r}} \leq C \sum_{i \leq k} (2^{-i}r)^{1/2} \leq C r^{1/2}.
        \end{equation}
        A similar iteration argument show that for all $s \in (0,r)$  it holds that \EEE
        \begin{equation}\label{eq_a_decay}
            \abs{A_{x_0,s} - A_{x_0,r}} \leq C \sum_{i \leq k} (2^{-i}r)^{-1/2} \leq C s^{-1/2}.
        \end{equation}
        By \eqref{eq_b_decay}, $b_{x_0,r}$ is a Cauchy sequence, and we see that the limit $b_{x_0} := \lim_{r \to 0} b_{x_0,r}$ exists, and in particular,
        \begin{equation}\label{eq_Bb}
            \lim_{r \to 0} \fint_{B(x_0,r)} \abs{b_{x_0,r} - b_{x_0}}^2 \dd{x} = 0.
        \end{equation}
Inequality \eqref{eq_a_decay} shows that $\abs{A_{x_0,r}} \leq C r^{-1/2} + \abs{A_{x_0,R}}$ for all $0 < r \leq R$, and in conjunction with (\ref{eq_Bb}), we have
        \begin{equation}\label{eq_Ba}
            \limsup_{r \to 0} \fint_{B(x_0,r)} \abs{a_{x_0,r} - b_{x_0}}^2 \dd{x} \leq C \limsup_{r \to 0} r^2 \abs{A_{x_0,r}}^2 = 0.
        \end{equation}
        Combining \eqref{eq_Bu} and \eqref{eq_Ba}, we see that
        \begin{equation*}
            \lim_{r \to 0} \fint_{B(x_0,r)} \abs{u - b_{x_0}}^2\dd{x} = 0.
        \end{equation*}
         This concludes the proof. \EEE
    \end{proof}


\typeout{References}

\end{document}